\documentclass[12pt]{article}
\usepackage[utf8]{inputenc}
\usepackage{amsmath,amsthm,amssymb}
\usepackage{graphicx,enumitem}
\usepackage{stmaryrd,microtype}
\usepackage[a4paper]{geometry}

\newtheorem{theorem}{Theorem}[section]
\newtheorem{lemma}[theorem]{Lemma}
\newtheorem{corollary}[theorem]{Corollary}
\newtheorem{proposition}[theorem]{Proposition}
\newtheorem{conjecture}[theorem]{Conjecture}
\newtheorem*{assertion*}{Assertion *}
\newtheorem{assertion}{Assertion}

\theoremstyle{definition}

\newtheorem{problem}[theorem]{Problem}

\theoremstyle{remark}

\newcommand{\vf}{\varphi}
\newcommand{\ovf}{f^{G,v}}
\newcommand{\abs}[1]{\left\lvert#1\right\rvert}
\newcommand{\sst}[2]{\left\{#1\,:\,#2\right\}}
\newcommand{\set}[1]{\left\{#1\right\}} 
\newcommand{\sset}[1]{\mathopen{\llbracket}#1\mathclose{\rrbracket}}

\newcommand{\intervalle}[4]{\mathopen{#1}#2\mathpunct{},#3\mathclose{#4}}
\newcommand{\icc}[2]{\intervalle{[}{#1}{#2}{]}}

\newcommand{\iset}[2]{\intervalle{\llbracket}{#1}{#2}{\rrbracket}}

\DeclareMathOperator{\I}{\mathcal{I}}
\DeclareMathOperator{\Pro}{Prob}
\DeclareMathOperator{\Ee}{E}

\renewcommand{\le}{\leqslant}
\renewcommand{\ge}{\geqslant}

\title{Subcubic triangle-free graphs have fractional chromatic number at most
$14/5$\thanks{%
This research was supported by the Czech-French Laboratory STRUCO and the Czech-French
bilateral project MEB 021115 (French reference PHC Barrande 24444XD)}}

\author{Zdeněk Dvořák\thanks{%
Computer Science Institute of Charles University, Prague, Czech Republic.
E-mail: \texttt{rakdver@iuuk.mff.cuni.cz}. Supported by the Center of
Excellence -- Inst. for Theor. Comp. Sci., Prague, project P202/12/G061 of
Czech Science Foundation.}
\and
Jean-Sébastien Sereni\thanks{%
CNRS (LORIA), Vand\oe uvre-lès-Nancy, France,
and Czech-French Laboratory STRUCO, CNRS and I\'UUK, Charles University,
Prague, Czech Republic.
E-mail: \texttt{sereni@kam.mff.cuni.cz}. This author's work was partially
supported by the French \emph{Agence Nationale de la Recherche} under reference
\textsc{anr 10 jcjc 0204 01}.}
\and
Jan Volec\thanks{%
Computer Science Institute of Charles University, Prague,
Czech Republic, and LIAFA, Univ. Paris Diderot, Paris, France.
Since January 2013, this author is a Ph.D. student at University of Warwick.
E-mail: \texttt{honza@ucw.cz}.
This author's work was supported by a grant of the French Government.}}

\date{}

\begin{document}

\maketitle

\begin{abstract}
We prove that every subcubic triangle-free graph has fractional chromatic number at most $14/5$,
thus confirming a conjecture of Heckman and Thomas~[A new proof of the
independence ratio of triangle-free cubic graphs. Discrete Math. 233 (2001), 233--237].
\end{abstract}

\section{Introduction}
One of the most celebrated results in Graph Theory is the Four-Color Theorem
(4CT). It states that every planar graph is 4-colorable.
It was solved by Appel and Haken~\cite{ApHa77,AHK77,ApHa89} in 1977 and,
about twenty years later, Robertson, Sanders, Seymour and Thomas~\cite{RSST97} found a new (and much simpler)
proof. However, both of the proofs require a computer assistance, and finding a fully
human-checkable proof is still one of the main open problems in Graph Theory.
An immediate corollary of the 4CT implies that every $n$-vertex planar graph
contains an independent set of size $n/4$ (this statement is sometimes called
the Erd\H os-Vizing conjecture). Although this seems to be an easier problem than
the 4CT itself, no proof without the 4CT is known. The best known result that
does not use the 4CT is due to Albertson~\cite{Alb76}, who showed the existence
of an independent set of size $2n/9$.

An intermediate step between the 4CT and the Erd\H os-Vizing conjecture is the
fractional version of the 4CT --- every planar graph is fractionally 4-colorable. In
fact, fractional colorings were introduced in 1973~\cite{HiRaSc73} as an
approach for either disproving, or giving more evidence to, the~4CT.  For a real
number $k$, a graph $G$ is fractionally $k$-colorable, if for every assignment
of weights to its vertices there is an independent set that contains at least
$(1/k)$-fraction of the total weight. In particular, every fractionally
$k$-colorable graph on $n$ vertices contains an independent set of size at
least $n/k$. The existence of independent sets of certain ratios in
\emph{subcubic} graphs, i.e., graphs with maximum degree at most $3$, led
Heckman and Thomas to pose the following two conjectures (a graph is called \emph{triangle-free} if it does not contain a triangle as
a subgraph).
\begin{conjecture}[Heckman and Thomas~\cite{HeTh01}]
\label{conjecture}
Every subcubic triangle-free graph is fractionally $14/5$-colorable.
\end{conjecture}
\begin{conjecture}[Heckman and Thomas~\cite{HeTh06}]
\label{conjecture-planar}
Every subcubic triangle-free planar graph is fractionally $8/3$-colorable.
\end{conjecture}
The purpose of this work is to establish Conjecture~\ref{conjecture}.
We believe that the method developed in this paper may be relevant for other fractional colouring
problems, and in particular for Conjecture~\ref{conjecture-planar}.

\subsection{History of the problem and related results}
Unlike for general planar graphs, colorings of triangle-free planar
graphs are well understood. Already in 1959, Grötzsch~\cite{Gro59}
proved that every triangle-free planar graph is $3$-colorable.
Therefore, such a graph on $n$ vertices has to contain an independent set of
size $n/3$. In 1976, Albertson, Bollob\'as and Tucker~\cite{ABT76} conjectured
that a triangle-free planar graph
also has to contain an independent set of size strictly larger than $n/3$.

Their conjecture was confirmed in 1993 by Steinberg and Tovey~\cite{StTo93},
even in a stronger sense: such a graph admits a 3-coloring where at least $\lfloor n/3 \rfloor + 1$
vertices have the same color. On the other hand, Jones~\cite{Jon84} found an
infinite family of triangle-free planar graphs with maximum degree four and
no independent set of size $\lfloor n/3 \rfloor + 2$.  However, if
the maximum degree is at most three, then Albertson et al.~\cite{ABT76}
conjectured that an independent set of size much larger than $n/3$ exists. Specifically,
they asked whether there is a constant $s \in \left(\frac{1}{3},\frac{3}{8} \right]$,
such that every subcubic triangle-free planar graph contains an independent set
of size $s n$. We note that for $s>3/8$ the statement would not be true, even
for graphs of girth five.

The strongest possible variant of this conjecture, i.e., for $s=3/8$, was finally
confirmed by Heckman and Thomas~\cite{HeTh06}. However, for $s=5/14$, it was
implied by a much earlier result of Staton~\cite{Sta79}, who actually showed that
every subcubic triangle-free (but not necessarily planar) graph contains an
independent set of size $5n/14$.  Jones~\cite{Jon90} then found a simpler proof
of this result; an even simpler one is due to Heckman and Thomas~\cite{HeTh01}. On
the other hand, Fajtlowicz~\cite{Faj78} observed that one cannot prove anything
larger than $5n/14$.  In 2009, Zhu~\cite{Zhu09} used an approach similar to that
of Heckman and Thomas to demonstrate that every $2$-connected subcubic
triangle-free $n$-vertex graph contains an induced bipartite subgraph of order
at least $5n/7$ except the Petersen graph and the dodecahedron --- thus
Staton's bound quickly follows.
As we already mentioned,
the main result of this paper is the strengthening of Staton's theorem to the fractional
(weighted) version, which was conjectured by Heckman and Thomas~\cite{HeTh01}.

This conjecture attracted a considerable amount of attention and it spawned a
number of interesting works in the last few years. In 2009, Hatami and
Zhu~\cite{HaZh09} showed that for every graph that satisfies the assumptions of
Conjecture~\ref{conjecture}, the fractional chromatic number is at most $3-3/64
\approx 2.953$. (The fractional chromatic number of a graph is the smallest
number $k$ such that the graph is fractionally $k$-colorable.)  The result of
Hatami and Zhu is the first to establish that the fractional chromatic number of
every subcubic triangle-free graph is smaller than $3$.  In 2012, Lu and
Peng~\cite{LuPe12} improved the bound to $3-3/43 \approx 2.930$.  There are
also two very recent improvements on the upper bound --- but with totally
different approaches. The first one is due to Ferguson, Kaiser and
Kr\'al'~\cite{FKK}, who showed that the fractional chromatic number is at most
$32/11 \approx 2.909$.  The other one is due to Liu~\cite{Liu}, who improved
the upper bound to $43/15 \approx 2.867$.

\section{Preliminaries}
\label{sec:prem}
We start with another definition of a fractional coloring that will be used in the paper.
It is equivalent to the one mentioned in the previous section by Linear Programming Duality;
a formal proof is found at the end of this section in Theorem~\ref{thm-eqchar}.
There are also another different (but equivalent) definitions of
a fractional coloring and the fractional chromatic number; for more details see,
e.g., the book of Scheinerman and Ullman~\cite{ScUl97}.

Let $G$ be a graph.
A \emph{fractional $k$-coloring} is an assignment of measurable subsets of the
interval $\icc{0}{1}$ to the vertices of $G$ such that each vertex is assigned
a subset of measure $1/k$ and the subsets assigned to adjacent vertices are
disjoint. The \emph{fractional chromatic number of~$G$} is the infimum over all
positive real numbers $k$ such that $G$ admits a fractional $k$-coloring. Note
that for finite graphs, such a real $k$ always exists, the infimum is in fact a
minimum, and its value is always rational. We let $\chi_f(G)$ be this minimum.

A \emph{demand function} is a function from $V(G)$ to
$\icc{0}{1}$ with rational values.
A \emph{weight function} is a function from $V(G)$ to the real numbers.  A weight function is \emph{non-negative}
if all its values are non-negative.  For a weight function $w$ and a set $X\subseteq V(G)$, let $w(X)=\sum_{v\in X} w(v)$.
For a demand function $f$, let $w_f=\sum_{v\in V(G)} f(v)w(v)$.

Let $\mu$ be the Lebesgue measure on real numbers.
An \emph{$f$-coloring} of $G$ is an assignment $\vf$ of measurable subsets of
$\icc{0}{1}$ to the vertices of $G$ such that
$\mu(\vf(v))\ge f(v)$ for every $v\in V(G)$ and such that $\vf(u)\cap \vf(v)=\varnothing$
whenever $u$ and $v$ are two adjacent vertices of $G$.
A positive integer $N$ is a \emph{common denominator} for $f$ if $N\cdot f(v)$ is an integer
for every $v\in V(G)$.  For integers $a$ and $b$, we define
$\iset{a}{b}$ to be the set $\{a,a+1,\ldots,b\}$, which is empty if $a>b$;
we set $\sset{a}=\iset{1}{a}$.
Let $N$ be a common denominator for $f$ and $\psi$ a function from $V(G)$ to
subsets of $\sset{N}$.
We say that $\psi$ is an \emph{$(f, N)$-coloring} of $G$ if $\abs{\psi(v)}\ge Nf(v)$ for every $v\in V(G)$ and
$\psi(u)\cap \psi(v)=\varnothing$ whenever $u$ and $v$ are adjacent vertices
of $G$.

Let us make a few remarks on these definitions.
\begin{itemize}
\item If $G$ has an $(f, N)$-coloring, then it also has an $(f, M)$-coloring for every $M$ divisible by $N$,
obtained by replacing each color by $M/N$ new colors.  Consequently, the
following statement, which is occasionally useful
in the proof, holds: if a graph $G_1$ has an $(f_1, N_1)$-coloring and a graph $G_2$ has an $(f_2, N_2)$-coloring, then there
exists an integer $N$ such that $G_1$ has an $(f_1, N)$-coloring and $G_2$ has an $(f_2, N)$-coloring.
\item For a rational number $r$, the graph $G$ has fractional chromatic number at most $r$ if and only if it has an $f_r$-coloring for the function $f_r$
that assigns $1/r$ to every vertex of $G$.  If $rN$ is an integer, then an $(f_r,N)$-coloring is usually called an $(rN:N)$-coloring in the literature.
\item In the definition of an $(f, N)$-coloring,
we can require that $\abs{\psi(v)}=Nf(v)$ for each vertex, as if $\abs{\psi(v)}>Nf(v)$, then we can remove colors from $\psi(v)$.
In particular, throughout the argument, whenever we receive an $(f,
N)$-coloring from an application of an inductive hypothesis,
we assume that the equality holds for every vertex.
\end{itemize}

To establish Theorem~\ref{thm-maingen}, we use several characterizations of
$f$-colorings. For a graph $G$, let $\I(G)$
be the set of all maximal independent sets.
Let \textsc{fracc} be the following linear program.

\begin{alignat*}{2}
    \text{Minimize: }    & \sum_{I \in \I(G)} x(I) \\
    \text{subject to: }  & \sum_{\substack{I \in \I(G)\\v\in I}} x(I) \ge f(v) &&
    \quad \text{for $v \in V(G)$;} \\
    & x(I) \ge 0 && \quad \text{for $I \in \I(G)$.}
\end{alignat*}
Furthermore, let \textsc{fracd} be the following program, which is the dual of
\textsc{fracc}.
\begin{alignat*}{2}
    \text{Maximize: }    & \sum_{v \in V(G)} f(v)\cdot y(v) \\
    \text{subject to: }  & \sum_{\phantom{G}v \in I\phantom{(G}} y(v) \le 1 &&
    \quad \text{for $I \in \I(G)$;} \\
    & y(v) \ge 0 && \quad \text{for $v \in V(G)$.}
\end{alignat*}
Notice that all the coefficients are rational numbers. Therefore, for both
programs there exist optimal solutions that are rational.
Moreover, since these two linear programs are dual of each other, the LP-duality
theorem ensures that they have the same value.
(The reader is referred to,
e.g., the book by Scheinerman and Ullman~\cite{ScUl97} for more details on
fractional graph theory.)

The following statement holds by standard arguments; the proof is included for
completeness.
\begin{theorem}\label{thm-eqchar}
Let $G$ be a graph and $f$ a demand function for $G$.  The following
statements are equivalent.
\begin{enumerate}
\item[\textrm{(a)}] The graph $G$ has an $f$-coloring.
\item[\textrm{(b)}] There exists a common denominator $N$ for $f$ such that $G$ has an $(f, N)$-coloring.
\item[\textrm{(c)}] For every weight function $w$, the graph $G$ contains an independent set $X$ such that $w(X)\ge w_f$.
\item[\textrm{(d)}] For every non-negative weight function $w$, the graph $G$ contains an independent set $X$ such that $w(X)\ge w_f$.
\end{enumerate}
\end{theorem}
\begin{proof}
Let us realize that (c) and (d) are indeed equivalent.  On the one hand,
(c) trivially implies (d).  On the other hand, let $w$ be a weight function. For each vertex
$v\in V(G)$, set $w'(v)=\max\{0,w(v)\}$. By (d), there
exists an independent set $I'$ of $G$ such that $w'(I')\ge\sum_{v\in
V(G)}f(v)w'(v)$. Setting $I=\sst{v\in I'}{w(v)>0}$ yields a (possibly empty)
independent set of $G$ with $w(I)\ge w_f$.  Hence, (d) implies (c).

We now prove that $\mathrm{(b)}\Rightarrow\mathrm{(a)}\Rightarrow\mathrm{(d)}\Rightarrow\mathrm{(b)}$.

\begin{description}[listparindent=15pt]
\item[$\mathrm{(b)}\Rightarrow\mathrm{(a)}$:] Assume that $\psi$ is an $(f,N)$-coloring of $G$, where $N$
is a common denominator for $f$. Setting
\[\varphi(v)=\bigcup_{i\in\psi(v)}\left[\frac{i-1}{N}\, ,\frac{i}{N}\right)\]
for each vertex $v\in V(G)$ yields an $f$-coloring of $G$.

\item[$\mathrm{(a)}\Rightarrow\mathrm{(d)}$:] Let $w$ be a non-negative weight function and
assume that $G$ has an $f$-coloring $\psi$. For each set $A\subseteq V(G)$, let
\[X(A)=\bigcap_{v\in A}\psi(v)\setminus \bigcup_{v\in V(G)\setminus A}\psi(v),\]
where $\bigcap_{v\in \varnothing}\psi(v)$ is defined to be $\icc{0}{1}$.
Note that the sets $X(A)\colon A\subseteq V(G)$ are pairwise disjoint and their union is $\icc{0}{1}$.
Let us choose a set $I\subseteq V(G)$ at random so that $\Pro[I=A]=\mu(X(A))$ for each $A\subseteq V(G)$.
Since $\psi$ is an $f$-coloring of $G$, we have $X(A)=\varnothing$ if $A$ is not an independent set, and thus $I$ is an independent set with probability $1$.
Furthermore, $\Pro[v\in I]=\sum_{\set{v}\subseteq A\subseteq V(G)} \mu(X(A))=\mu(\psi(v))\ge f(v)$ for each $v\in V(G)$.
We conclude that
\begin{align*}
\Ee[w(I)]=&\sum_{v\in V(G)}\Pro[v\in I]w(v)\\
\ge&\sum_{v\in V(G)}f(v)w(v)=w_f.
\end{align*}
Therefore, there exists $I\in\I(G)$ with $w(I)\ge w_f$.

\item[$\mathrm{(d)}\Rightarrow\mathrm{(b)}$:] We proceed in two steps. First, we show that, assuming (d),
the value of \textsc{fracc} is at most $1$. Next, we infer the existence of
an $(f,N)$-coloring of $G$ for a common denominator $N$ of $f$.

Let $b$ be the value of \textsc{fracd} and let $y$ be a
corresponding solution.  Note that $y$ is a non-negative weight function for $G$, and thus by (d),
there exists an independent set $I$ of $G$ such that $y(I)\ge y_f=b$. Since $y$ is a feasible solution of \textsc{fracd},
we deduce that $b\le 1$.

By the LP-duality theorem, \textsc{fracd} and \textsc{fracc} have the same value.
Let $x$ be a rational feasible solution of \textsc{fracc} with value at most~$1$.
Fix a common denominator $N$ for $f$ and $x$. An $(f,N)$-coloring $\psi$ of $G$ can
be built as follows. Set
$\I'=\sst{I\in\I(G)}{x(I)>0}$ and let $I_1,\ldots,I_k$ be
the elements of $\I'$. For each $i\in\{1,\ldots,k\}$, set
\[T_i=\biggl\llbracket1+N\cdot\sum_{j=1}^{i-1}x(I_j)\,
,N\cdot\sum_{j=1}^{i}x(I_j)\biggr\rrbracket.\]
Observe that
$\abs{\bigcup_{i=1}^{k}T_i}=\sum_{i=1}^{k}\abs{T_i}=N\cdot\sum_{i=1}^kx(I_i)\le N$.
For each vertex $v\in V(G)$, let $\I(v)=\sst{i\in\sset{k}}{v\in I_i}$
and define $\psi(v)=\bigcup_{i\in\I(v)} T_i$.

The obtained function $\psi$ is
an $(f,N)$-coloring of $G$. Indeed, for each vertex $v\in V(G)$ we have
$\abs{\psi(v)}\ge N\cdot\sum_{i\in\I(v)}x(I_i)\ge N f(v)$.
Moreover, if $u$ and $v$ are two vertices adjacent in $G$, then
$\I(u)\cap\I(v)=\varnothing$ and, consequently,
$\psi(u)\cap\psi(v)=\varnothing$.
\end{description}
\end{proof}

\section{The proof}
We commonly use the following observation.
\begin{proposition}\label{prop-basic}
Let $f$ be a demand function for a graph $G$, let $N$ be a common denominator for $f$ and let $\psi$ be an $(f,N)$-coloring for $G$.
\begin{enumerate}
\item If $xyz$ is a path in $G$, then $\abs{\psi(x)\cup\psi(z)}\le (1-f(y))N$. Equivalently,
$\abs{\psi(x)\cap\psi(z)}\ge (f(x)+f(z)+f(y)-1)N$.
\item If $xvyz$ is a path in $G$, then $\abs{\psi(x)\cap\psi(z)}\le (1-f(v)-f(y))N$.
\end{enumerate}
Conversely, if $f(a)+f(b)\le 1$ for each edge $ab$ of the path and $\psi$ is an $(f,N)$-coloring of $x$ and $z$
satisfying the conditions 1. and 2. above, then $\psi$ can be extended to an $(f,N)$-coloring of the path $xyz$ or $xvyz$,
respectively.
\end{proposition}

\begin{figure}[!t]
\begin{center}
\includegraphics[width=80mm]{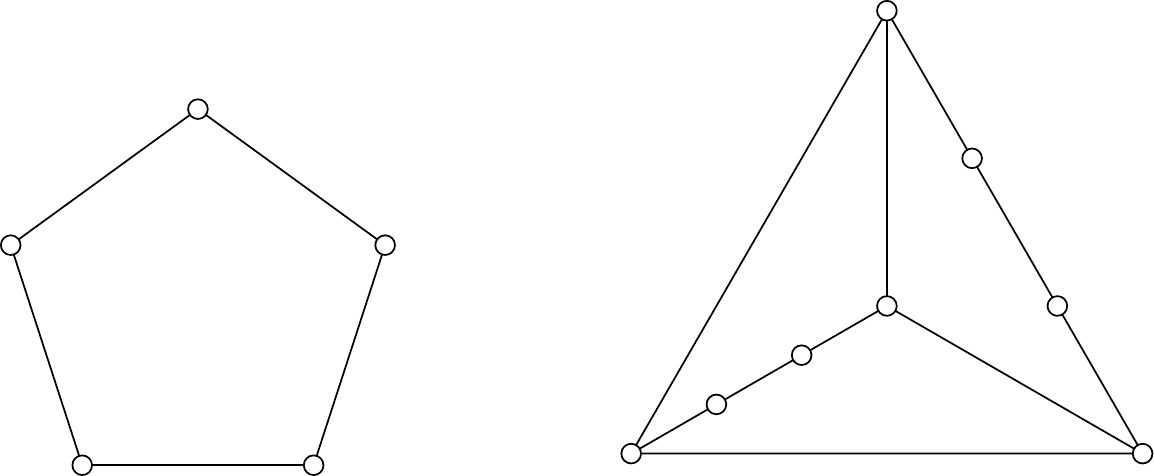}
\end{center}
\caption{Dangerous graphs.}\label{fig-dang}
\end{figure}

A graph $H$ is \emph{dangerous} if $H$ is either a $5$-cycle or the graph $K'_4$ obtained from $K_4$ by subdividing both edges of
its perfect matching twice, see Figure~\ref{fig-dang}.  The vertices of degree two of a dangerous graph are called \emph{special}.
Let $G$ be a subcubic graph and let $B$ be a subset of its vertices.  Let $H$ be a dangerous induced subgraph of $G$.
A special vertex $v$ of $H$ is \emph{$B$-safe} if either $v\in B$ or $v$ has degree three in $G$.
If $B$ is empty, we write just ``safe'' instead of ``$\varnothing$-safe''.
If $G$ is a subcubic graph, a set $B\subseteq V(G)$ is called a \emph{nail} if every vertex in $B$ has degree at most two
and every dangerous induced subgraph of $G$ contains at least two $B$-safe special vertices.  For a subcubic graph $G$ and
its nail $B$, let $f^G_B$ be the demand function defined as follows: if $v\in B$, then $f^G_B(v)=(7-\deg_G(v))/14$;
otherwise $f^G_B(v)=(8-\deg_G(v))/14$.  When the graph $G$ is clear from the context, we drop the superscript and write just $f_B$
for this demand function.

In order to show that every subcubic triangle-free graph has fractional chromatic number at most $14/5$, we prove the following
stronger statement.
\begin{theorem}\label{thm-maingen}
If $G$ is a subcubic triangle-free graph and $B\subseteq V(G)$ is a nail, then $G$ has an $f_B$-coloring.
\end{theorem}

We point out that the motivation for the formulation of Theorem~\ref{thm-maingen} as well as for some parts of its proof
comes from the work of Heckman and Thomas~\cite{HeTh01}, in which an
analogous strengthening is used to prove that every
subcubic triangle-free graph on $n$ vertices contains an independent set of
size at least $5n/14$.

A subcubic triangle-free graph $G$ with a nail $B$ is a \emph{minimal counterexample to Theorem~\ref{thm-maingen}}
if $G$ has no $f_B$-coloring, and for every subcubic triangle-free graph $G'$ with a nail $B'$ such that
either $\abs{V(G')}<\abs{V(G)}$, or $\abs{V(G')}=\abs{V(G)}$ and $\abs{B'}<\abs{B}$, there exists an $f_{B'}$-coloring of $G'$.
The proof proceeds by contradiction, showing that there is no minimal counterexample to Theorem~\ref{thm-maingen}.
Let us first study the properties of such a hypothetical minimal counterexample.
\begin{lemma}\label{lemma-conn}
If a subcubic triangle-free graph $G$ with a nail $B$ is a minimal counterexample to Theorem~\ref{thm-maingen},
then $G$ is $2$-edge-connected.
\end{lemma}
\begin{proof}
Clearly, $G$ is connected.  Suppose that $uv\in E(G)$ is a bridge, and let $G_1$ and $G_2$ be the components of $G-uv$
such that $u\in V(G_1)$ and $v\in V(G_2)$.  Let $B_1=(B\cap V(G_1))\cup\{u\}$ and $B_2=(B\cap V(G_2))\cup\{v\}$.
Note that $B_1$ is a nail for $G_1$ and $B_2$ is a nail for $G_2$, and thus by the minimality of $G$,
there exist a common denominator $N$ for $f_{B_1}$ and $f_{B_2}$, an $(f_{B_1}, N)$-coloring $\psi_1$ for $G_1$ and
an $(f_{B_2}, N)$-coloring $\psi_2$ for $G_2$.  Since $u\in B_1$ and $v\in
B_2$, we have $f_{B_1}^{G_1}(u)\le 7/14$
and $f_{B_2}^{G_2}(v)\le 7/14$, thus we can assume (by permuting the colors in $\psi_2$ if necessary) that
$\psi_1(u)$ and $\psi_2(v)$ are disjoint. It follows that the union of $\psi_1$ and $\psi_2$ is
an $(f_B, N)$-coloring of $G$, contrary to the assumption that $G$ is a counterexample.
\end{proof}
\begin{lemma}\label{lemma-mindeg2}
If a subcubic triangle-free graph $G$ with a nail $B$ is a minimal counterexample to Theorem~\ref{thm-maingen},
then $G$ has minimum degree at least two.
\end{lemma}
\begin{proof}
Suppose, on the contrary, that $v$ is a vertex of degree at most one in $G$.
Since $G$ is $2$-edge-connected by Lemma~\ref{lemma-conn}, it follows that
$v$ has degree $0$ and $V(G)=\{v\}$.  However, $\vf(v)=\icc{0}{1}$ is then an $f_B$-coloring of $G$, since $\mu(\vf(v))=1>f_B(v)$.
This contradicts the assumption that $G$ is a counterexample.
\end{proof}
\begin{figure}[!t]
\begin{center}
\includegraphics[width=100mm]{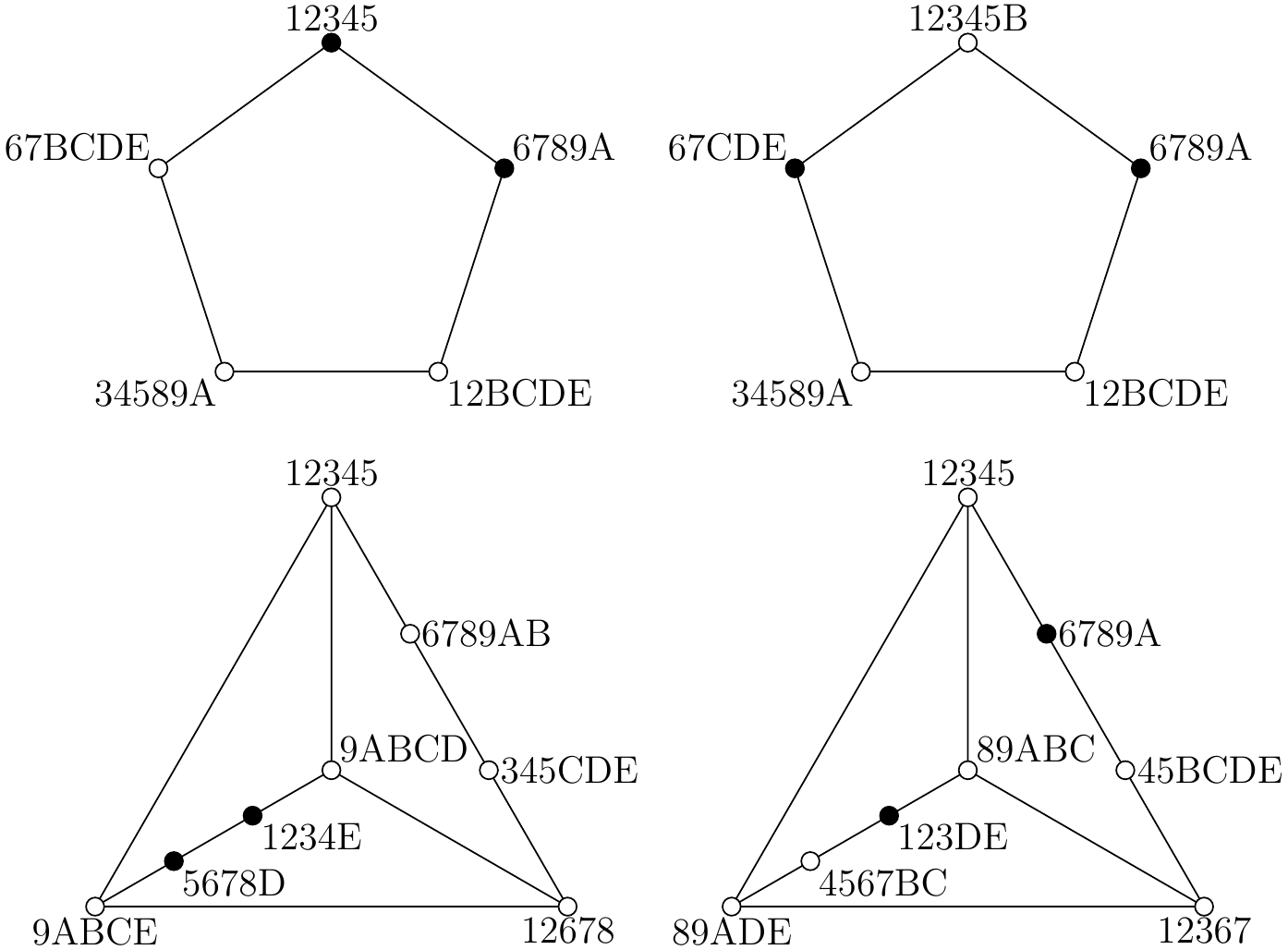}
\end{center}
\caption{Colorings of dangerous graphs with minimal nails.  The nails consist of the black vertices.}\label{fig-coldang}
\end{figure}
\begin{lemma}\label{lemma-nob}
If a subcubic triangle-free graph $G$ with a nail $B$ is a minimal counterexample to Theorem~\ref{thm-maingen},
then $B=\varnothing$.
\end{lemma}
\begin{proof}
Suppose, on the contrary, that $B$ contains a vertex $b$.  If $B'=B\setminus\{b\}$ were a nail in $G$, then by the minimality of $G$ and $B$,
there would exist an $f_{B'}$-coloring of $G$, which would also be an $f_B$-coloring of $G$.  Therefore,
we can assume that $G$ contains a dangerous induced subgraph $H$ with at most one $B'$-safe vertex.  Since $G$ is
$2$-edge-connected by Lemma~\ref{lemma-conn}, it follows that $G=H$.  Consequently, $B$ consists of exactly two
special vertices of $G$.  However, Figure~\ref{fig-coldang} shows all possibilities for $G$ and $B$ up to isomorphism
together with their $(f_B,14)$-colorings, contradicting the assumption that $G$ is a counterexample.
\end{proof}

In view of the previous lemma, we say that a subcubic triangle-free graph $G$ is a minimal counterexample to Theorem~\ref{thm-maingen}
if the empty set is a nail for $G$ and together they form a minimal counterexample to Theorem~\ref{thm-maingen}.
\begin{lemma}\label{lemma-22ip}
Let $G$ be a minimal counterexample to Theorem~\ref{thm-maingen}.  If $u$ and $v$ are adjacent vertices of $G$ of degree two,
then there exists a $5$-cycle in $G$ containing the edge $uv$.
\end{lemma}
\begin{proof}
Suppose, on the contrary, that $uv$ is not contained in a $5$-cycle.  Let $x$
and $y$ be the neighbors of $u$ and $v$, respectively, that are not in $\{u,v\}$.
Note that $x\neq y$ since $G$ is triangle-free.  Let $G'$ be the graph obtained from $G-\{u,v\}$ by adding the edge $xy$.
Since the edge $uv$ is not contained in a $5$-cycle, it follows that $G'$ is triangle-free.

If the empty set is a nail for $G'$, then by the minimality of $G$, there
exists an $(f_\varnothing, 14t)$-coloring $\psi'$ of $G'$ for a positive integer $t$.
The sets $\psi'(x)$ and $\psi'(y)$ are disjoint; by permuting the colors if necessary, we can assume that
$\psi'(x)\subseteq \sset{6t}$ and $\psi'(y)\subseteq \iset{6t+1}{12t}$.  Then, there exists an $(f_\varnothing, 14t)$-coloring $\psi$
of $G$, defined by $\psi(z)=\psi'(z)$ for $z\not\in\{u,v\}$, $\psi(u)=\iset{6t+1}{12t}$ and $\psi(v)=\sset{6t}$.
This contradicts the assumption that $G$ is a counterexample.

We conclude that $\varnothing$ is not a nail for $G'$.
Thus if $x$ and $y$ are adjacent in $G$, then both these vertices have degree
$3$ in $G$ since $G$ is a minimal counterexample.
Therefore, the very same argument as above using $\{x,y\}$ as a
nail for $G'$ yields an $f_\varnothing$-coloring for $G$, a contradiction.

As observed earlier, $G'$ contains a dangerous induced subgraph $H$ with at most one safe special
vertex.  Lemma~\ref{lemma-conn} implies that $G$ is $2$-edge-connected, and thus $G'$ is $2$-edge-connected as well.  It follows that $G'=H$.
Consequently, since $x$ and $y$ are not adjacent, $G$ is one of the graphs depicted in
Figure~\ref{fig-colsubdang}, which are exhibited together with an $(f_\varnothing,14)$-coloring.
This is a contradiction.
\end{proof}
\begin{figure}[th]
\begin{center}
\includegraphics[width=100mm]{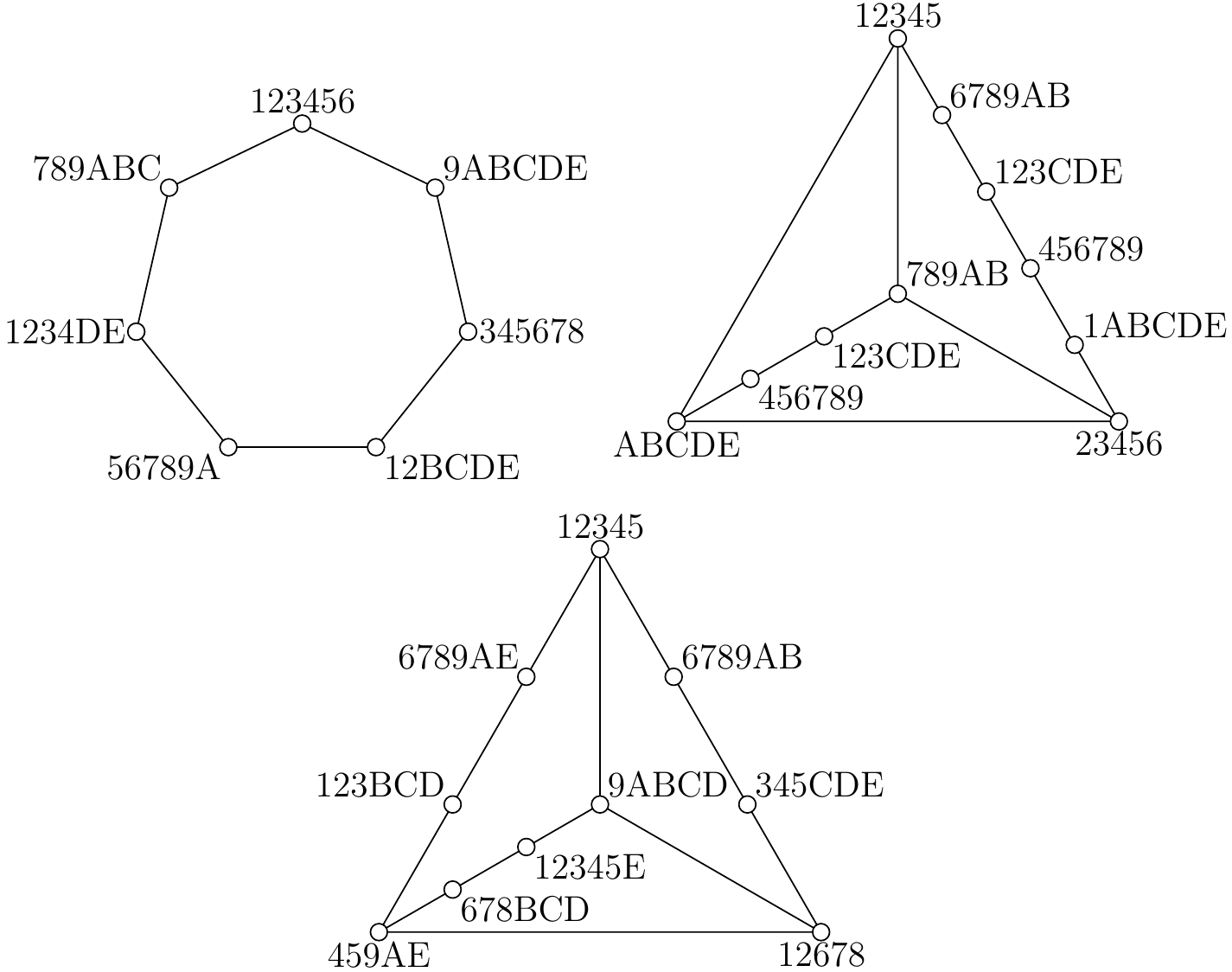}
\end{center}
\caption{Colorings of subdivided dangerous graphs.}\label{fig-colsubdang}
\end{figure}

\begin{lemma}\label{lemma-3conn}
Let $G$ be a minimal counterexample to Theorem~\ref{thm-maingen}.  If $\{uv,xy\}$ is an edge-cut
in $G$ and $G_1$ and $G_2$ are connected components of $G-\{uv,xy\}$, then
$\min\{\abs{V(G_1)},\abs{V(G_2)}\}\le 2$.
\end{lemma}
\begin{proof}
Suppose, on the contrary, that $\min\{\abs{V(G_1)},\abs{V(G_2)}\}\ge 3$.  Choose the labels so that
$\{u,x\}\subset V(G_1)$.

Suppose first that
$G_1$ is a path $uzx$ on three vertices.  By Lemma~\ref{lemma-22ip}, the vertices $y$ and $v$ are adjacent.
Since $\abs{V(G_2)}\ge3$ and $G$ is $2$-edge-connected by Lemma~\ref{lemma-conn},
it follows that $y$ and $v$ have degree three in $G$.
Note that $B'=\{y,v\}$ is a nail for $G_2$.  By the minimality
of $G$, there exists an $(f_{B'},14t)$-coloring $\psi$ of $G_2$ for a positive integer $t$.  Since $y$ and $v$ are adjacent,
by permuting the colors, we can assume that $\psi(y)=\sset{5t}$ and $\psi(v)=\iset{5t+1}{10t}$.  Let us extend $\psi$
by defining $\psi(u)=\sset{2t}\cup\iset{10t+1}{14t}$, $\psi(z)=\iset{2t+1}{8t}$ and
$\psi(x)=\iset{8t+1}{14t}$.
Then $\psi$ is an $(f_\varnothing,14t)$-coloring of $G$, contrary to the assumption that $G$ is a counterexample.

By symmetry, we conclude that neither $G_1$ nor $G_2$ is a path on three vertices; and more generally, neither $G_1$ nor
$G_2$ is a path, as otherwise $G$ would contain a $2$-edge-cut cutting off a path on three vertices.
Therefore, we can choose the edge-cut $\{uv,xy\}$ in such a way that both $x$ and $v$ have degree three.
Let $G'_1$ be the graph obtained from $G_1$ by adding a path $uabx$, and let $G'_2$ be the
graph obtained from $G_2$ by adding a path $vcdy$, where $a$, $b$, $c$ and $d$ are new vertices of degree two.
Since $G$ is $2$-edge-connected, we have $u\neq x$ and $v\neq y$; hence, both $G'_1$ and $G'_2$ are triangle-free.
If $y$ has degree three, then let $B_1=\{a,b\}$, otherwise let $B_1=\varnothing$.  Similarly, if $u$ has degree three,
then let $B_2=\{c,d\}$, otherwise let $B_2=\varnothing$.

Suppose first that $B_1$ is a nail for $G'_1$ and $B_2$ is a nail for $G'_2$.  By the minimality of $G$, there exist
an $(f_{B_1},14t)$-coloring $\psi_1$ of $G'_1$ and an $(f_{B_2},14t)$-coloring
$\psi_2$ of $G'_2$, for a positive integer $t$.
Let $n_u=\abs{\psi_1(u)\setminus \psi_1(x)}$, $n_x=\abs{\psi_1(x)\setminus \psi_1(u)}$, $n_{ux}=\abs{\psi_1(u)\cap \psi_1(x)}$, and let $n_v$, $n_y$ and $n_{vy}$ be defined symmetrically.
Proposition~\ref{prop-basic} implies that $n_{ux}\le 4t$ and $n_{vy}\le 4t$.  Since $x$ and $v$ have degree three and
$u$ and $y$ have degree at least two, it follows that $n_x+n_{ux}=5t$, $n_u+n_{ux}\le6t$, $n_v+n_{vy}=5t$ and $n_y+n_{vy}\le 6t$.
Furthermore, by the choice of $B_1$ and $B_2$, either $n_u+n_{ux}=5t$ or $n_{vy}\le 2t$, and either $n_y+n_{vy}=5t$ or $n_{ux}\le 2t$.
Therefore,
\begin{equation}\label{ineq-twosets}
n_{ux}+n_{vy}+\max(n_u,n_y)+\max(n_v,n_x)\le 14t.
\end{equation}
Consequently, we can permute the colors for $\psi_2$
so that the sets $\psi_1(u)\cap \psi_1(x)$, $\psi_2(v)\cap \psi_2(y)$, $(\psi_1(u)\setminus \psi_1(x))\cup (\psi_2(y)\setminus \psi_2(v))$
and $(\psi_1(x)\setminus \psi_1(u))\cup (\psi_2(v)\setminus \psi_2(y))$ are pairwise disjoint.
Indeed, by~\eqref{ineq-twosets} the interval $[14t]$ can be partitioned
into four intervals $I_1,I_2,I_3,I_4$ with $\abs{I_1}=n_{ux}$,
$\abs{I_2}=n_{vy}$, $\abs{I_3}=\max\{n_u,n_y\}$ and
$\abs{I_4}\ge\max\{n_v,n_x\}$. Now, we
permute the colors for $\psi_1$ so that $\psi_1(u)\cap \psi_1(x)=I_1$,
$\psi_1(u)\setminus \psi_1(x) \subseteq I_3$
and $\psi_1(x)\setminus \psi_1(u) \subseteq I_4$.
In addition, we permute the colors for $\psi_2$ so that $\psi_2(v)\cap
\psi_2(y)=I_2$, $\psi_2(v)\setminus \psi_2(y) \subseteq I_4$
and $\psi_2(y)\setminus \psi_2(v) \subseteq I_3$.
Then, $\psi_1(u)\cap \psi_2(v)=\varnothing=\psi_1(x)\cap\psi_2(y)$, thus giving an $(f_\varnothing, 14t)$-coloring of $G$, which is a contradiction.

Hence, we can assume that $B_1$ is not a nail for $G'_1$.  Since $G$ is $2$-edge-connected, $G'_1$ is $2$-edge-connected as well, and thus
it is a dangerous graph.  Since $G_1$ is not a path on three vertices, it
follows that $G'_1$ is $K'_4$.  Furthermore, $B_1=\varnothing$ and
thus $y$ has degree two.  Note that $G_1$ has an
$(f_{\set{u,x}},14t)$-coloring such that $n_{ux}=4t$ and $n_u=n_x=t$ (obtained from the
coloring of the bottom left graph in Figure~\ref{fig-coldang} by removing the
black vertices and replacing each color $c$ by $t$ new colors $c_1,\ldots,c_t$).
Let $G''_2$ be the graph obtained from $G_2$ by adding a new vertex of degree two adjacent to $y$ and $v$.
Let us point out that $y$ is not adjacent to $v$, since $G$ is
$2$-edge-connected (recall that $y$ has degree two since $B_1=\varnothing$).
Hence, $G''_2$ is triangle-free.
If $\varnothing$ is a nail for $G''_2$, then let us redefine $\psi_2$ as an
$(f_\varnothing, 14t)$-coloring of $G''_2$, which
exists by the minimality of $G$, and let $n_v$, $n_y$ and $n_{vy}$ be defined as before.
Proposition~\ref{prop-basic} yields that
$n_v+n_y+n_{vy}\le 8t$; hence,~\eqref{ineq-twosets} holds, and we obtain a contradiction as in the previous paragraph.

\begin{figure}[!t]
\begin{center}
\includegraphics[width=100mm]{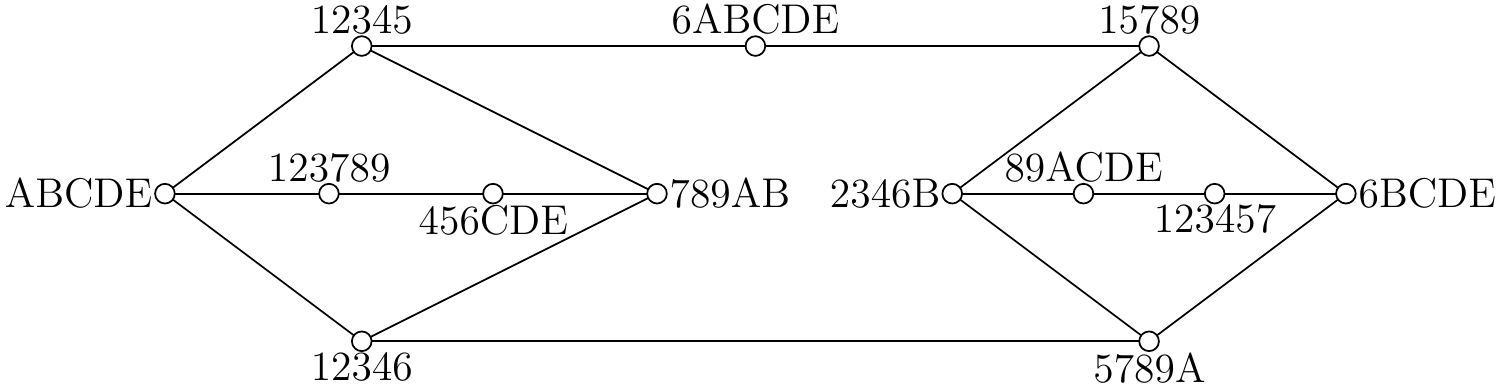}
\end{center}
\caption{A special $2$-cut.}\label{fig-joink4}
\end{figure}

Consequently, $\varnothing$ is not a nail for $G''_2$, and since $G''_2$ is $2$-edge-connected and $G_2$ is not a path,
it follows that $G''_2$ is $K'_4$.  However, $G$ must then be the graph depicted in Figure~\ref{fig-joink4} together with its
$(f_\varnothing, 14)$-coloring, contrary to the assumption that $G$ is a counterexample.
\end{proof}

\begin{corollary}\label{cor-3safe}
Every dangerous induced subgraph in a minimal counterexample to Theorem~\ref{thm-maingen} contains at least three safe special vertices.
\end{corollary}
\begin{proof}
Let $H$ be a dangerous induced subgraph in a minimal counterexample $G$.
Since $\varnothing$ is a nail for $G$, it follows that $H$ contains
at least two safe special vertices $u$ and $v$.  If $H$ contains exactly two safe special vertices, then the edges
of $E(G)\setminus E(H)$ incident with $u$ and $v$ form a $2$-edge-cut.
By Lemma~\ref{lemma-3conn}, we know that $G$ consists
of $H$ and a path $Q$ of length two or three joining $u$ and $v$.  Note that $u$ and $v$ are not adjacent, as otherwise
$Q$ would either be part of a triangle or contradict Lemma~\ref{lemma-22ip}.  If $H$ is a $5$-cycle, then
$G$ has an $(f_\varnothing,14)$-coloring obtained from the coloring of the top right graph in Figure~\ref{fig-coldang} by copying
the colors of the vertices of one of the paths between the black vertices to
the vertices of $Q$.  Hence, we assume that $H$ is $K'_4$.  By Lemma~\ref{lemma-22ip},
we conclude that $Q$ has length two.  Consequently, $G$ is the graph depicted in Figure~\ref{fig-3safe}.  However,
this graph has an $(f_\varnothing,14)$-coloring, which is a contradiction.
\end{proof}

\begin{figure}[!t]
\begin{center}
\includegraphics[width=60mm]{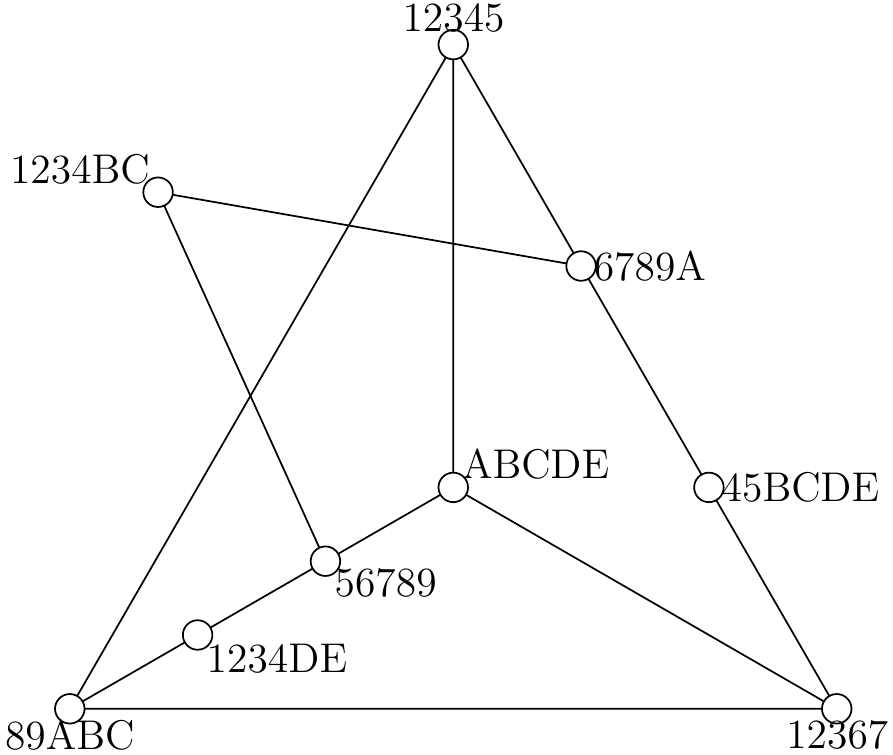}
\end{center}
\caption{A dangerous induced subgraph with two safe vertices.}\label{fig-3safe}
\end{figure}

\begin{lemma}\label{lemma-no22}
If $G$ is a minimal counterexample to Theorem~\ref{thm-maingen}, then no two vertices of $G$ of degree two are adjacent.
\end{lemma}
\begin{proof}
Suppose, on the contrary, that $u$ and $v$ are adjacent vertices of degree two in $G$.
It follows from Lemma~\ref{lemma-22ip} that $G$ contains a $5$-cycle
$xuvyz$.  Further, $x$, $y$ and $z$ have degree three by Corollary~\ref{cor-3safe}.
Let $a$, $b$ and $c$ be the neighbors of $x$, $y$ and $z$, respectively, outside of the $5$-cycle
(where possibly $a=b$).

Let us now consider the case where $a$ has degree two. Note that, in this
case, $a\neq b$ as $G$ is $2$-edge-connected. Let $d$ be the neighbor of $a$ distinct from $x$.
If $d$ has degree two, then by Lemma~\ref{lemma-22ip}, the path $xad$ is a part of a $5$-cycle.  Since $G$ is $2$-edge-connected,
it follows that $d$ is adjacent to $c$.  Then $G$ contains a $2$-edge-cut
formed by the edges incident with $y$ and $c$.
By Lemma~\ref{lemma-3conn}, $G$ is one of the graphs in the top of Figure~\ref{fig-only2}.  This is a contradiction,
as the figure also shows that these graphs are $(f_\varnothing, 14)$-colorable.  Hence, $d$ has degree three.
Let $G'=G-\{u,v\}$ and $B'=\{x,y\}$.  Then $B'$ is a nail for $G'$.  By the minimality of $G$, there exists an $(f_{B'},14t)$-coloring $\psi'$ of $G'$ for
a positive integer $t$.  Let $L=\sset{14t}\setminus \psi'(z)$. Note that $\abs{L}=9t$ and $\psi'(y)\subseteq L$.
Since the path $daxz$ is colored and $f_{B'}(a)=6/14$ and $f_{B'}(x)=5/14$, Proposition~\ref{prop-basic}
implies that $\abs{\psi'(d)\cap\psi'(z)}\le 3t$, and thus $\abs{\psi'(d)\cap L}\ge 2t$.  We construct an $(f_\varnothing, 14t)$-coloring $\psi$ of $G$ as follows.
We let $\psi$ be equal to $\psi'$ on all vertices but $a$, $x$, $u$ and $v$.
Let $M$ be a subset of $\psi(d)\cap L$ of size exactly $2t$.
Let $M'$ be a subset of $\psi'(y)$ of size exactly $2t$ containing $M\cap\psi'(y)$.
We choose $\psi(x)$ of size $5t$ so that $M\cup M'\subset \psi(x)\subseteq M\cup M'\cup (L\setminus \psi'(y))$.
Observe that $\abs{\psi(x)\cap\psi(d)}\ge \abs{M}=2t$ and $\abs{\psi(x)\cap \psi(y)}=\abs{M'}=2t$;
hence, Proposition~\ref{prop-basic} implies that we can choose $\psi(a)$, $\psi(u)$ and $\psi(v)$ so that
$\psi$ is an $(f_\varnothing, 14t)$-coloring of $G$.  This is a contradiction.

By symmetry, it follows that both $a$ and $b$ have degree three.  If $a=b$,
then the edges incident with $a$ and $z$ form a $2$-edge-cut in $G$,
so Lemma~\ref{lemma-3conn} yields that $G$ consists of the $5$-cycle $xuvyz$, the vertex $a$ adjacent to $x$ and $y$, and a path $Q$ of length two or three
joining $a$ with $z$.  If $Q$ had length three, then $G$ would be $K'_4$, contrary to the assumption that $\varnothing$ is a nail for $G$.
So $Q$ has length two and hence $G$ is the bottom graph in Figure~\ref{fig-only2}, which has an $(f_\varnothing, 14)$-coloring.
This is a contradiction; hence, $a\neq b$.

\begin{figure}[!t]
\begin{center}
\includegraphics[width=100mm]{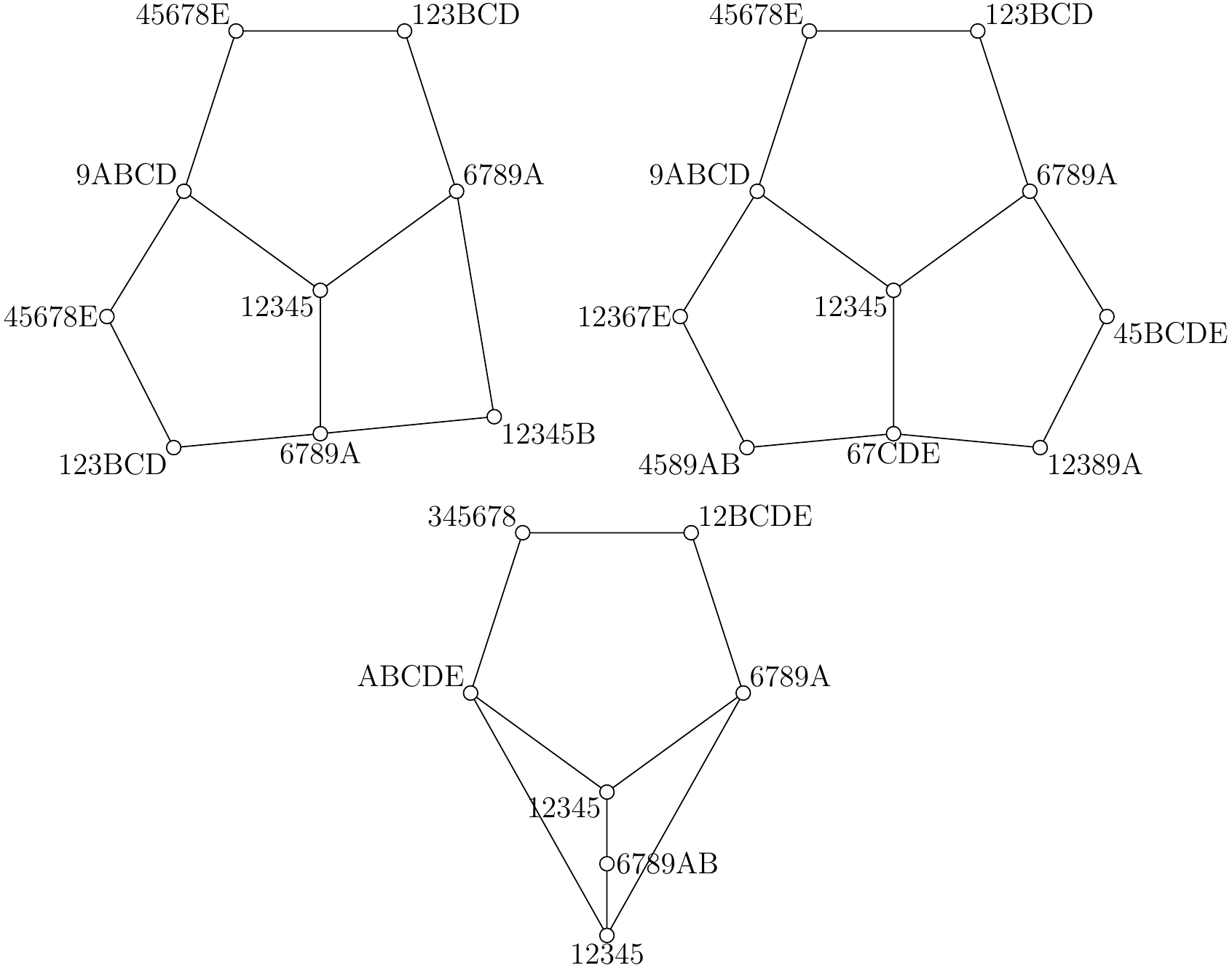}
\end{center}
\caption{Vertices of degree $2$ in a $5$-cycle.}\label{fig-only2}
\end{figure}

Suppose now that $c$ has degree two, and let $s$ be the neighbor of $c$
distinct from $z$.  If $s$ has degree two, then using Lemma~\ref{lemma-22ip} and symmetry,
we can assume that $s$ is adjacent to $a$.  Then the edges incident with $a$
and $y$ form a $2$-edge-cut. However, this
contradicts Lemma~\ref{lemma-3conn} since $b$ has degree three. Hence, $s$ has degree three.  Let $G'$ be the graph obtained from $G-\{u,v,x,y,z,c\}$ by adding a path
$aopb$ with two new vertices of degree two. Note that $B'=\{o,p,s\}$ is a nail for $G'$.
By the minimality of $G$, there exists an $(f_{B'},14t)$-coloring $\psi$ of
$G'$ for a positive integer $t$.  Let $L_x=\sset{14t}\setminus \psi(a)$
and $L_y=\sset{14t}\setminus \psi(b)$.  Thus, $\abs{L_x}=\abs{L_y}=9t$, and Proposition~\ref{prop-basic} applied to the path $aopb$ implies that
$\abs{L_x\cup L_y}\ge 10t$.  Since $\abs{L_x\cup L_y}\le 14t$, we also know
that $\abs{L_x\cap L_y}\ge 4t$.  Choose $M$ as an arbitrary subset of $\psi(s)$ of size exactly $2t$.
Observe that we can choose $\psi(x)$ in $L_x\setminus M$ and $\psi(y)$ in
$L_y\setminus M$, each of size $5t$, so that
$\abs{\psi(x)\cap\psi(y)}=2t$: first choose a set $M'$  of $2t$ colors in
$(L_x\cap L_y)\setminus M$; next, choose disjoint sets of size $3t$ from $L_x\setminus (M\cup M')$ and $L_y\setminus (M\cup M')$,
which is possible as each of these sets has size at least $3t$ (in fact, at
least $5t$) and their union has size at least $6t$.
Notice that $\abs{(\psi(x)\cup \psi(y))\cap \psi(s)}\le \abs{\psi(s)\setminus M}\le 3t$.  By Proposition~\ref{prop-basic},
$\psi$ extends to an $(f_\varnothing, 14t)$-coloring of $G$ (to color $z$ and $c$, apply the Proposition to a path of length three
with ends colored by $\psi(s)$ and $\psi(x)\cup \psi(y)$), which is a contradiction.

Therefore, $c$ has degree three.  Let $G'=G-\{u,v,y\}$.  Suppose first that $\varnothing$ is a nail for $G'$.  By the minimality
of $G$, there exists an $(f^{G'}_\varnothing,14t)$-coloring $\psi'$ of $G'$
for a positive integer $t$.  Let $L_x=\sset{14t}\setminus \psi'(a)$,
$L_y=\sset{14t}\setminus \psi'(b)$ and $L_z=\sset{14t}\setminus \psi'(c)$.  Note that $\abs{L_x}=\abs{L_z}=9t$ and $\abs{L_y}=8t$.  Also, Proposition~\ref{prop-basic}
implies that $\abs{L_x\cup L_z}\ge 12t$.  Arbitrarily choose a set $M$ in $L_z\setminus L_x$ of size exactly $3t$.
Note that $\abs{L_z\setminus M}=6t$ and $\abs{L_y\setminus M}\ge 5t$; hence,
there exists a set $Z$ in $L_z\setminus M$ of size exactly
$2t$ such that $\abs{L_y\setminus (M\cup Z)}\ge 4t$.  Let $Y$ be a subset of $L_y\setminus (M\cup Z)$ of size exactly $4t$.
If $\abs{Z\setminus L_x}\ge t$, then let $Y'=\varnothing$; otherwise notice that
$\abs{L_x\cup Z\cup M}<13t$ and choose $Y'$ in $\sset{14t}\setminus (L_x\cup Z\cup M)$
of size exactly $t$. Last, choose a set $T$ of size $3t$ so that $Y'\subset T\subset Y\cup Y'$.

Let $\psi$ be an $(f_\varnothing, 14t)$-coloring of $G$ defined as follows.  We set $\psi(p)=\psi'(p)$ for $p\in V(G)\setminus \{x,y,z,u,v,b\}$,
$\psi(z)=M\cup Z$, $\psi(u)=M\cup T$ and we let $\psi(y)$ be any set of $5t$ colors such that $Y\cup Y'\subset \psi(y)\subset \sset{14t}\setminus (M\cup Z)$.
Thus $\abs{\psi(u)\cap \psi(y)}\ge \abs{T}=3t$, hence we can choose $\psi(v)$
in $\sset{14t}\setminus (\psi(u)\cup \psi(y))$ of size $6t$.
The choice of $Y'$ and $T$ implies that either $\abs{Z\setminus L_x}\ge t$ or
$\abs{T\setminus L_x}\ge t$; hence
$\abs{L_x\setminus (\psi(u)\cup \psi(z))}=\abs{L_x\setminus (T\cup Z)}\ge
\abs{L_x}-\abs{T}-\abs{Z}+t=5t$.  Choose a set $\psi(x)$ in $L_x\setminus (\psi(u)\cup \psi(z))$
of size exactly $5t$. Also note that $\abs{\psi(y)\setminus L_y}\le \abs{\psi(y)\setminus Y}=t$, and select
$\psi(b)\subseteq \psi'(b)\setminus (\psi(y)\setminus L_y)$ of size $5t$ (let
us point out that $f^{G'}_\varnothing(b)=6/14$ while
$f^G_\varnothing(b)=5/14$).  The existence of the coloring $\psi$ contradicts the assumption that $G$ is a counterexample.

\begin{figure}[!t]
\begin{center}
\includegraphics[width=50mm]{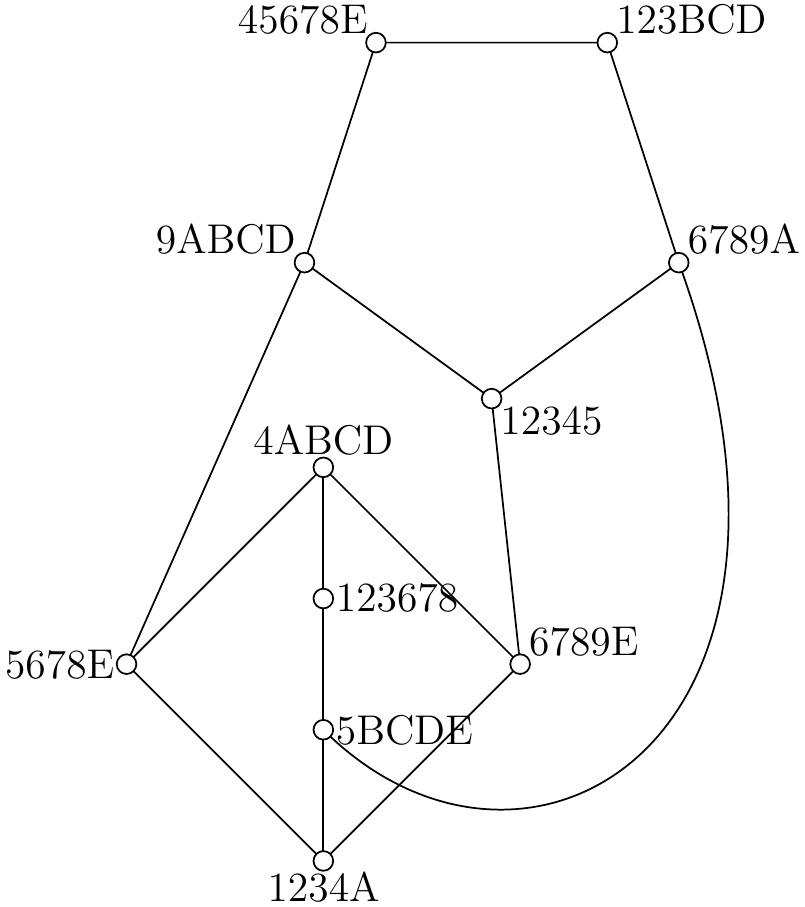}
\end{center}
\caption{Configuration from Lemma~\ref{lemma-no22}.}\label{fig-22k4}
\end{figure}

Finally, let us consider the case that $\varnothing$ is not a nail for $G'$.  Therefore, $G'$ contains a dangerous induced subgraph $H$ with at most
one safe special vertex.  By Corollary~\ref{cor-3safe}, $H$ has at least three special vertices that are safe in $G$.  It follows that
$H$ contains at least two of $x$, $z$ and $b$.  In particular, $H$ contains
$x$ or $z$, and since $x$ and $z$ have degree two in $G'$, we infer that $H$ contains
both of them.  Since $a$ and $c$ have degree three in $G'$, we deduce that $H$ is $K'_4$.  Let $s_1$ and $s_2$ be the special vertices of $H$
distinct from $x$ and $z$.  If both $s_1$ and $s_2$ have degree three in $G$, then since $\varnothing$ is not a nail for $G'$,
one of them, $s_i$, is adjacent to $y$ (that is, $s_i=b$); it follows that
$s_{3-i}$ is incident with a bridge in $G$, contrary to Lemma~\ref{lemma-conn}.
Hence, we can assume that $s_2$ has degree two in $G$.  By
Corollary~\ref{cor-3safe}, the vertex $s_1$ has degree three in $G$.
Recalling that $b$ also has degree three in $G$, we infer that
either $G$ is the graph depicted in Figure~\ref{fig-22k4}, or G has a
$2$-edge-cut formed by the edge $yb$ and one of the edges incident to $s_1$.
The latter case is excluded as it would contradict Lemma~\ref{lemma-3conn}, since
$b$ has degree three in $G$. The former case would imply that $G$ is
$(f_\varnothing,14)$-colorable,
as demonstrated in Figure~\ref{fig-22k4}. This contradiction concludes the proof.
\end{proof}

\begin{lemma}\label{lemma-nokp4}
No minimal counterexample to Theorem~\ref{thm-maingen} contains $K'_4$ as an induced subgraph.
\end{lemma}
\begin{proof}
Suppose, on the contrary, that a minimal counterexample $G$ contains $K'_4$ as an induced subgraph. That is,
$G$ contains a $4$-cycle $uvxy$ of vertices of degree three together with paths $uabx$ and $vcdy$.
By Corollary~\ref{cor-3safe}, we can assume that $b$, $c$ and $d$ have degree three in $G$.

Suppose first that we can choose the subgraph so that $a$ has degree two.
Let $b'$ be the neighbor of $b$ distinct
from $a$ and $x$.  Since we consider $K'_4$ as an induced subgraph of $G$, we have $c\neq b'\neq d$.  Let $G'=G-\{u,v,x,y,a,b\}$ and $B'=\{c,d,b'\}$.
Since $B'$ is a nail for $G'$, the minimality of $G$ implies that there exists an $(f_{B'}, 14t)$-coloring $\psi$ of $G'$
for a positive integer $t$.  By permuting the colors, we can assume that
$\psi(c)=\sset{5t}$ and $\psi(d)=\iset{5t+1}{10t}$.

Note that $\abs{\psi(b')}\le 6t$.  To extend $\psi$ to an
$(f_\varnothing,14t)$-coloring of $G$,
it suffices to show that one can choose sets $\psi(b), \psi(v), \psi(y)\subset \sset{14t}$ of size $5t$
disjoint from
$\psi(b')$, $\psi(c)$ and $\psi(d)$, respectively, in such a way that
$\abs{\psi(v)\cap\psi(y)}=4t$,
$\abs{(\psi(v)\cup \psi(y))\cup \psi(b)}=9t$ and $\abs{(\psi(v)\cup \psi(y))\cap \psi(b)}=2t$.
Indeed, if this is possible, then $\psi$ can be further extended to $a$,
$u$ and $x$ by Proposition~\ref{prop-basic}, which contradicts the assumption
that $G$ is a counterexample. It remains to show why the aforementioned sets
exist.
We consider two cases. First,
if $\abs{\psi(b')\cap\iset{10t+1}{14t}}\le 2t$,
then choose $\psi(b)$ in $\sset{14t}\setminus \psi(b')$ of size $5t$ so that
$\abs{\psi(b)\cap \iset{10t+1}{14t}}=2t$; furthermore, choose $\psi(v)$ and $\psi(y)$
of size $5t$ so that they are disjoint with $\psi(c)$ and $\psi(d)$,
respectively, and satisfy $\psi(v)\cap \psi(y)=\iset{10t+1}{14t}$ and
$(\psi(v)\cup \psi(y))\cap \psi(b)\subset \iset{10t+1}{14t}$.
Second, if $\abs{\psi(b')\cap\iset{10t+1}{14t}}>2t$, then note that $\abs{\psi(b')\cap \sset{10t}}<4t$; hence, we can choose
$\psi(b)$ in $\sset{10t}\setminus \psi(b')$ of size $5t$ so that $\abs{\psi(b)\cap
\sset{5t}}\ge t$ and $\abs{\psi(b)\cap \iset{5t+1}{10t}}\ge t$;
next, we choose $\psi(v)$ and $\psi(y)$ of size $5t$ so that they are disjoint
from $\psi(c)$ and $\psi(d)$, respectively, and satisfy $\psi(v)\cap
\psi(y)=\iset{10t+1}{14t}$ and
$(\psi(v)\cup \psi(y))\cap \sset{10t}\subset \psi(b)$.

The contradiction that we obtained in the previous paragraph shows that $a$ cannot have degree two.
Consequently, we can assume that for every occurrence of $K'_4$ as an induced subgraph in $G$,
all the special vertices are safe.  Let $G'=G-\{u,v,x,y\}$ and suppose first that $\varnothing$ is a nail
for $G'$. Then, the minimality of $G$ ensures that there exists an
$(f^{G'}_\varnothing,14t)$-coloring $\psi'$ of $G'$ for a positive integer $t$.
Let $L_u=\sset{14t}\setminus \psi(a)$, $L_x=\sset{14t}\setminus \psi(b)$, $L_v=\sset{14t}\setminus \psi(c)$ and $L_y=\sset{14t}\setminus \psi(d)$,
and note that $\abs{L_u}=\abs{L_v}=\abs{L_x}=\abs{L_y}=8t$.
By Tuza and Voigt~\cite[Theorem 2]{TuVo96}, there exist sets $A_u\subset L_u$, $A_v\subset L_v$, $A_x\subset L_x$ and $A_y\subset L_y$
such that $\abs{A_u}=\abs{A_v}=\abs{A_x}=\abs{A_y}=4t$ and $A_x\cup A_u$ is
disjoint from $A_y\cup A_v$.  Let $M_u=\sset{14t}\setminus (A_u\cup A_v\cup A_y)$,
$M_v=\sset{14t}\setminus (A_v\cup A_u\cup A_x)$, $M_x=\sset{14t}\setminus (A_x\cup A_v\cup A_y)$ and $M_y=\sset{14t}\setminus (A_y\cup A_u\cup A_x)$.
Each of these sets having size at least $2t$, applying again the result of Tuza and Voigt~\cite{TuVo96}, we infer
the existence of sets $B_u\subset M_u$, $B_v\subset M_v$, $B_x\subset M_x$ and $B_y\subset M_y$
such that $\abs{B_u}=\abs{B_v}=\abs{B_x}=\abs{B_y}=t$ and $B_x\cup B_u$ is
disjoint from $B_y\cup B_v$.  Let $\psi$ be defined as follows:
$\psi(z)=\psi'(z)$ for $z\in V(G)\setminus \{a,b,c,d,u,v,x,y\}$, $\psi(a)=\psi'(a)\setminus B_u$, $\psi(b)=\psi'(b)\setminus B_x$, $\psi(c)=\psi'(c)\setminus B_v$, $\psi(d)=\psi'(d)\setminus B_y$,
$\psi(u)=A_u\cup B_u$, $\psi(v)=A_v\cup B_v$, $\psi(x)=A_x\cup B_x$ and $\psi(y)=A_y\cup B_y$.  Then $\psi$ is an $(f^G_\varnothing,14t)$-coloring of $G$
(notice that $f^{G'}_\varnothing(z)=6/14$, while $f^G_\varnothing(z)=5/14$ whenever
$z\in\{a,b,c,d\}$).  This contradicts the assumption that $G$ is a counterexample.

\begin{figure}[!t]
\begin{center}
\includegraphics[width=50mm]{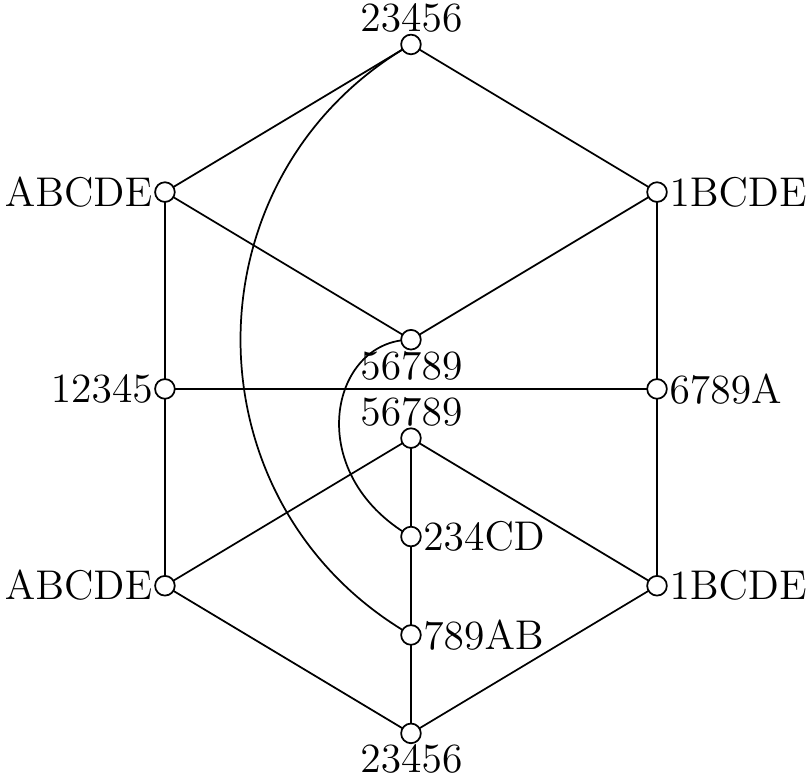}
\end{center}
\caption{A configuration from Lemma~\ref{lemma-nokp4}.}\label{fig-k4plus}
\end{figure}

Finally, it remains to consider the case where $G'$ contains a dangerous induced subgraph $H$ with at most one safe special vertex.
As $H$ contains at least two safe vertices in $G$, we can assume by symmetry
that $H$ contains $a$. Since $a$ has degree two in $G'$, the subgraph $H$ contains $b$ as well.
Suppose now that $H$ also contains at least one of $c$ and $d$ (and thus both
of them). Then $H$ must be isomorphic to $K_4'$. Indeed, since the subgraph of $G$ induced by
$\{u,v,x,y,a,b,c,d\}$ is isomorphic to $K'_4$, it follows that $\{a,b,c,d\}$
induces a matching, and thus $H$ cannot be a $5$-cycle.
We conclude that $H$ is isomorphic to $K'_4$, $G'=H$ and $G$ is the graph depicted in Figure~\ref{fig-k4plus}.  However, then $G$ is
$(f_\varnothing, 14)$-colorable, which is a contradiction.

Hence, neither $c$ nor $d$ belongs to $H$, and thus $H$ contains a special vertex that is unsafe in $G$.
Since the case that $G$ contains $K'_4$ with an unsafe special vertex has already been
excluded, it follows that $H$ is a $5$-cycle
$abb'sa'$, where (by Lemma~\ref{lemma-no22}) $a'$ and $b'$ have degree two and $s$ has degree three.  Let $G_1=G-\{u,v,x,y,a,b,a',b'\}$
and $B_1=\{c,d\}$.  Note that $s$ has degree $1$ in $G_1$.
As $B_1$ is a nail for $G_1$, the
minimality of $G$ ensures the existence of an $(f_{B_1},14t)$-coloring
$\psi_1$ of $G_1$ for a positive integer $t$.
By permuting the colors, we can assume that $\psi(c)=\sset{5t}$ and
$\psi(d)=\iset{5t+1}{10t}$.  Let us extend $\psi$ to $G$ as follows. First, we
delete from $\psi(s)$ an arbitrary subset of $2t$ colors, so that $\psi(s)$
has now size $5t$.
Set $\psi(v)=\iset{9t+1}{14t}$ and $\psi(y)=\sset{t}\cup\iset{10t+1}{14t}$.
Arbitrarily choose disjoint sets
$M_a$ in $\psi(s)\setminus \iset{9t+1}{12t}$ and $M_b$ in $\psi(s)\setminus
\sset{t}\cup\iset{12t+1}{14t}$
each of size $2t$.
Choose two disjoint subsets $\psi(a)$ and $\psi(b)$ of $\sset{14t}$, each of size
$5t$, so that $M_a\cup \sset{t}\cup\iset{12t+1}{14t}\subseteq \psi(a)$ and $M_b\cup
\iset{9t+1}{12t}\subseteq \psi(b)$.
Note that $\abs{\iset{t+1}{9t}\setminus \psi(z)}\ge 6t$ for $z\in\{a,b\}$;
hence, we can choose for $\psi(u)$ and $\psi(x)$ two sets of size $5t$, both
in $\iset{t+1}{9t}$
and disjoint from $\psi(a)$ and $\psi(b)$, respectively.
Furthermore, note that $\abs{\psi(a)\cup \psi(s)}\le 8t$ and $\abs{\psi(b)\cup \psi(s)}\le 8t$.  It follows that $\psi$ can be extended to $a'$ and $b'$ by Proposition~\ref{prop-basic}.
The obtained mapping $\psi$ is an $(f_\varnothing, 14t)$-coloring of $G$, which is a contradiction.
\end{proof}

\begin{lemma}\label{lemma-22ipm}
Let $G$ be a minimal counterexample to Theorem~\ref{thm-maingen}.  Let $v$ be
a vertex of $G$ and let $x$ and $y$ be two neighbors of $v$.
Suppose that $x$ and $y$ have degree two, and let $x'$ and $y'$ be their neighbors, respectively, distinct from $v$.  Then $x'\neq y'$
and $x'$ is adjacent to $y'$.
\end{lemma}
\begin{proof}
The vertices $v$, $x'$ and $y'$ have degree three by Lemma~\ref{lemma-no22}.
If $x'=y'$, then let $G'=G-x$ and $B'=\{x',v\}$.  Since $B'$ is a nail for
$G'$, the minimality of $G$ ensures that there
exists an $f^{G'}_{B'}$-coloring $\psi$ of $G'$.  We can extend $\psi$ to an $f^G_\varnothing$-coloring of $G$ by setting $\psi(x)=\psi(y)$,
contradicting the assumption that $G$ is a counterexample.

Therefore, $x'\neq y'$.
Let $u$ be the neighbor of $v$ distinct from $x$ and $y$.
Our next goal is to prove that $u$ must
have degree three. Suppose, on the contrary, that $u$ has degree two, and let $u'$ be the neighbor
of $u$ distinct from $v$. Then $u'$ has degree $3$ and, by symmetry, we infer
that $x'\neq u'\neq y'$.  Let $G'=G-\{u,v,x,y\}$ and let $B'=\{u',x',y'\}$.
Since $B'$ is a nail for $G'$, the minimality of $G$ implies the existence of
an $(f^{G'}_{B'},14t)$-coloring $\psi$ of $G'$ for a positive integer $t$.
Note that $\abs{\psi(u')}=\abs{\psi(x')}=\abs{\psi(y')}=5t$.
For $i\in\{1,2,3\}$, let $S_i$ be the set of elements of $\sset{14t}$ that belong to exactly $i$ of the sets $\psi(u')$, $\psi(x')$ and $\psi(y')$.
Note that $\abs{S_1}+\abs{S_2}+\abs{S_3}\le 14t$ and $\abs{S_1}+2\abs{S_2}+3\abs{S_3}=15t$,
so $\abs{S_2}+2\abs{S_3}\ge t$.  Let $M\subset S_2\cup S_3$ be an arbitrary set
such that $\abs{M\cap S_2}+2\abs{M\cap S_3}\ge t$ and $\abs{M}\le t$.  Choose $M_u\subset\psi(u')\setminus M$, $M_x\subset\psi(x')\setminus M$ and $M_y\subset\psi(y')\setminus M$
arbitrarily so that $\abs{M\cap\psi(u')}+\abs{M_u}=\abs{M\cap\psi(x')}+\abs{M_x}=\abs{M\cap\psi(y')}+\abs{M_y}=2t$, and let $L=M\cup M_u\cup M_x\cup M_y$.
Thus
\begin{align*}
\abs{L}&\le \abs{M}+\abs{M_u}+\abs{M_x}+\abs{M_y}\\
&=6t+\abs{M}-\abs{M\cap\psi(u')}-\abs{M\cap\psi(x')}-\abs{M\cap\psi(y')}\\
&=6t+\abs{M}-2\abs{M\cap S_2}-3\abs{M\cap S_3}\\
&=6t-\abs{M\cap S_2}-2\abs{M\cap S_3}\\
&\le 5t.
\end{align*}
Let us choose $\psi(v)$ in $\sset{14t}$ of size $5t$ such that $L\subseteq
\psi(v)$.  Note that $\abs{\psi(v)\cap\psi(z)}\ge 2t$ for $z\in\{u',x',y'\}$;
hence, $\psi$ can be extended to $u$, $x$ and $y$ by Proposition~\ref{prop-basic}.
This yields an $f_\varnothing$-coloring of $G$, which is a contradiction.  Therefore, $u$ has degree three.

Now suppose, for a contradiction, that $x'$ is not adjacent to $y'$ in $G$.
Then, the graph $G'$ obtained from $G$ by removing $x$ and adding the edge
$x'y$ is triangle-free.  Let us show that $\varnothing$ is a nail for $G'$.
Consider a dangerous induced subgraph $H$ of $G'$.  If $H$ had at most one safe special vertex in $G'$, then $G'$ would
contain two adjacent vertices $a$ and $b$ of degree two.  Note that $v$ is the only vertex of $G'$ of degree two that has
degree three in $G$, and that both neighbors of $v$ in $G'$ have degree three.  It follows that $a$ and $b$ have degree two
in $G$ as well.  Furthermore, $y$ has degree three in $G'$, thus the edge $ab$ is distinct from $x'y$.  Therefore,
$a$ and $b$ would be adjacent vertices of degree two in $G$, contrary to Lemma~\ref{lemma-no22}.

By the minimality of $G$, there exists an $(f^{G'}_\varnothing,14t)$-coloring $\psi'$ of $G'$
for a positive integer $t$. Let us show that
$\abs{(\psi'(x')\cup\psi'(y'))\cap \psi'(u)}\le 3t$.
Indeed, Proposition~\ref{prop-basic} applied to the path $uvy$ ensures that
$\abs{\psi'(u)\cap\psi'(y)}\ge 2t$. Thus, as $\abs{\psi'(u)}=5t$, it follows that
$\abs{\psi'(u)\setminus\psi'(y)}\le 3t$. Noting that $\psi'(y)$ is disjoint from each of
$\psi'(x')$ and $\psi'(y')$, we see that
$\psi'(u) \cap (\psi'(x')\cup\psi'(y'))$ is contained in
$\psi'(u)\setminus\psi'(y)$, which yields the announced inequality.

Choose arbitrary sets $M_x$ in $\psi'(x')\setminus \psi'(u)$ and $M_y$ in
$\psi'(y')\setminus \psi'(u)$, each of size $2t$.
We define a coloring $\psi$ of $G$ as follows. Set $\psi(z)=\psi'(z)$ for each $z\in V(G)\setminus \{x,y,v\}$.
Choose $\psi(v)$ in $\sset{14t}\setminus\psi(u)$ of size $5t$ so that $M_x\cup M_y\subset \psi(v)$.
It holds that $\abs{\psi(x')\cap \psi(v)}\ge \abs{M_x}=2t$
and $\abs{\psi(y')\cap \psi(v)}\ge 2t$; hence, $\psi$ can be extended to $x$
and $y$ by Proposition~\ref{prop-basic}.
Observe that $\psi$ is an $(f^G_\varnothing,14t)$-coloring of $G$, which is a contradiction.
\end{proof}

\begin{lemma}\label{lemma-no2t2}
If $G$ is a minimal counterexample to Theorem~\ref{thm-maingen}, then every vertex of $G$ has at most one neighbor of degree two.
\end{lemma}
\begin{proof}
Suppose, on the contrary, that a vertex $v$ of $G$ has two distinct neighbors $x$ and
$y$ of degree two in $G$.
Let $x'$ and $y'$ be the neighbors of $x$ and $y$, respectively, distinct from $v$.
Lemma~\ref{lemma-22ipm} implies that $vxx'y'y$ is a $5$-cycle.
Moreover, Lemma~\ref{lemma-no22} implies that $x'$, $y'$ and $v$ all have degree three.  Let $u$ be the neighbor of $v$ distinct from $x$ and $y$.  If $u$
had degree two, then by Lemma~\ref{lemma-22ipm}, its neighbor distinct from $v$ would be adjacent both to $x'$ and $y'$, and $G$ would contain
a triangle.  Hence, $u$ has degree three.  Let $a$ and $b$ be the neighbors of $x'$ and $y'$, respectively, not belonging to the path $xx'y'y$
(where possibly $a=u$ or $b=u$).

\begin{figure}[!t]
\begin{center}
\includegraphics[width=100mm]{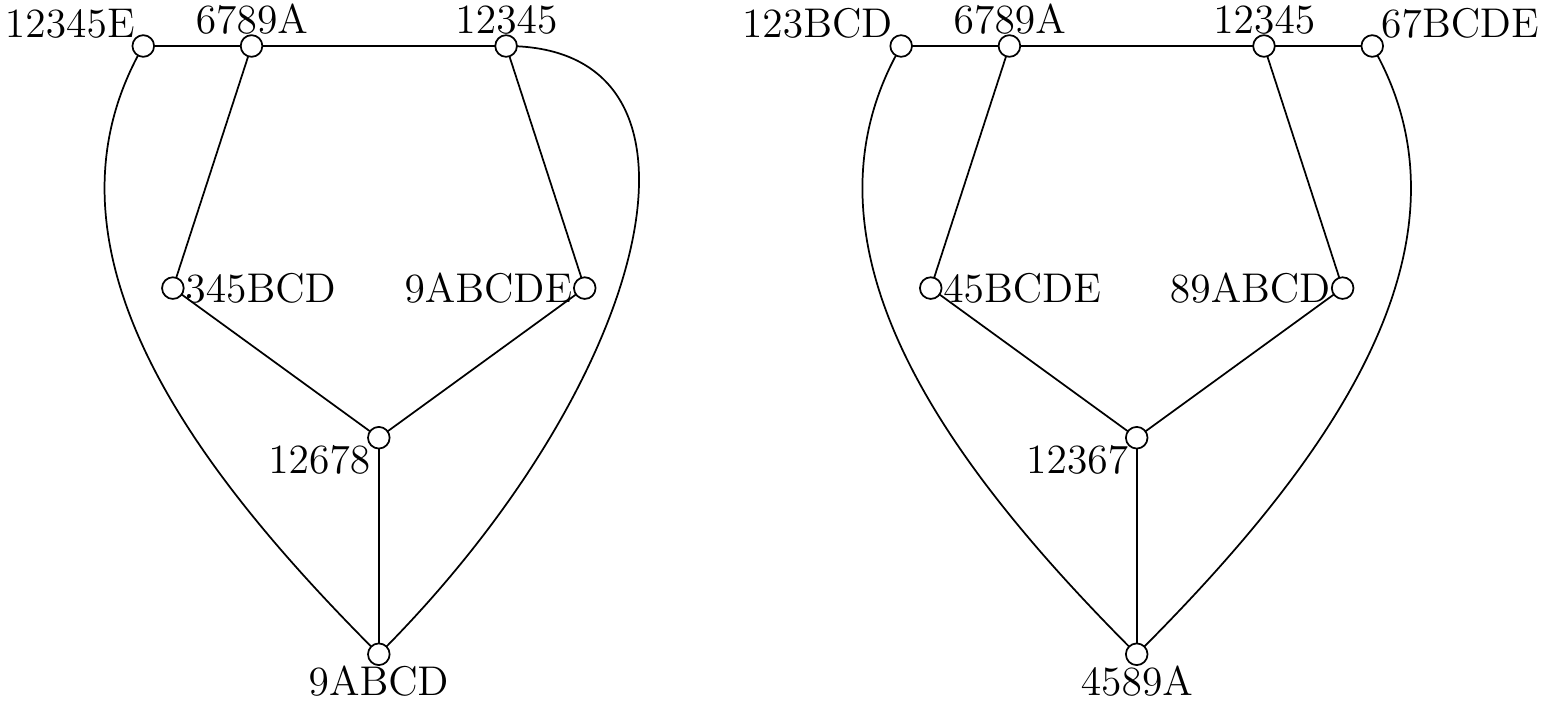}
\end{center}
\caption{Configurations from Lemma~\ref{lemma-no2t2}.}\label{fig-2t2}
\end{figure}

If $a$ has degree two, then Lemma~\ref{lemma-22ipm} yields that $a$ is adjacent to $u$.  By Lemmas~\ref{lemma-3conn} and \ref{lemma-no22},
it follows that either $b=u$, or $b$ has degree two and is adjacent to $u$ as
well.  However, $G$ would then be one of the graphs in Figure~\ref{fig-2t2},
which are both $(f_\varnothing, 14)$-colorable.  Therefore, $a$ has degree three
and, by symmetry, so does $b$.

Let $G'=G-\{x,y,v\}$ and $B'=\{x',y',u\}$.  Since $B'$ is a nail for $G'$, the minimality of $G$
ensures the existence of an $(f_{B'},14t)$-coloring $\psi'$ of $G'$ for a
positive integer $t$.
Let $L_v=\sset{14t}\setminus \psi'(u)$.  As $\abs{L_v}=9t$ and
$\abs{\psi'(a)}=\abs{\psi'(b)}=5t$, we can choose disjoint sets
$M_a$ in $L_v\setminus \psi'(a)$ and $M_b$ in $L_v\setminus \psi'(b)$ each of size $2t$.
We define a coloring $\psi$ of $G$ as follows.  For $z\in V(G)\setminus \{v,x,x',y,y'\}$, set $\psi(z)=\psi'(z)$.
Proposition~\ref{prop-basic} yields that $\abs{\psi'(a)\cap \psi'(b)}\le 4t$,
and thus we can choose $\psi(x')$ in $\sset{14t}\setminus (\psi'(a)\cup M_b)$ of size $5t$
so that $M_a\subset \psi(x')$ and $\abs{(\psi'(b)\setminus \psi'(a))\cap\psi(x')}\ge t$.  Let $L_{y'}=\sset{14t}\setminus (\psi(x')\cup \psi(b))$.
Note that $M_b\subset L_{y'}$ and $\abs{L_{y'}}\ge 5t$.  Choose
$\psi(y')$ in $L_{y'}$ of size $5t$ so that $M_b\subset \psi(y')$,
and $\psi(v)$ in $L_v$ of size $5t$ so that $M_a\cup M_b\subset \psi(v)$.
It follows that $\abs{\psi(v)\cap \psi(x')}\ge \abs{M_a}=2t$ and
$\abs{\psi(v)\cap \psi(y')}\ge\abs{M_b}=2t$; hence $\psi$ can be extended to $x$ and $y$
as well, by Proposition~\ref{prop-basic}.
However, $\psi$ is then an $(f_\varnothing, 14t)$-coloring of $G$, which is a contradiction.
\end{proof}

The following is a direct consequence of Lemmas~\ref{lemma-no22} and
\ref{lemma-no2t2}.
\begin{corollary}\label{cor-4safe}
In a minimal counterexample to Theorem~\ref{thm-maingen}, every $5$-cycle
contains at least four safe vertices.
\end{corollary}

We continue our study of the structure of minimal counterexamples that contain
vertices of degree two.
\begin{lemma}\label{lemma-col2}
Let $G$ be a minimal counterexample to Theorem~\ref{thm-maingen} and let $v\in V(G)$ have degree two.
Let $x$ and $y$ be the neighbors of $v$; let the neighbors of $x$ distinct from $v$ be $a$ and $b$,
and let the neighbors of $y$ distinct from $v$ be $c$ and $d$.  Then the
following hold.
\begin{enumerate}
\item\label{it:1} $\varnothing$ is a nail for $G-v$, as well as for $G-\{v,x,y\}$.
\item\label{it:2} The vertices $a$, $b$, $c$ and $d$ are pairwise distinct.
\item\label{it:3} We let $\ovf$ be the function defined by $\ovf(z)=f_\varnothing(z)$ for $z\in V(G)\setminus \{v,x,y,a,b,c,d\}$,
$\ovf(z)=4/14$ for $z\in\{a,b,c,d\}$, $\ovf(x)=\ovf(y)=8/14$ and $\ovf(v)=2/14$.  Then $G$ has an $\ovf$-coloring.
\end{enumerate}
\end{lemma}
\begin{proof}
Note that $a$, $b$, $c$ and $d$ have degree three by Lemma~\ref{lemma-no2t2}.
Let us consider each part of the statement separately.

\ref{it:1}. Let $G'$ be either $G-v$ or $G-\{v,x,y\}$ and suppose that $H$ is a dangerous induced subgraph
of $G'$ containing at most one safe vertex.
Lemma~\ref{lemma-nokp4} implies that $H$ is a $5$-cycle. By Corollary~\ref{cor-4safe},
at least four of its vertices are safe in $G$.  It follows that $H$ contains at least
three vertices that have degree two in $G'$ and degree three in $G$.  There are only two such vertices if
$G'=G-v$.  Hence, we assume that $G'=G-\{v,x,y\}$.  If all vertices of $H$ are safe in $G$, then
$a,b,c,d\in V(H)$; however, the vertex of $H$ distinct from $a$, $b$, $c$ and
$d$ is then incident with a bridge in $G$, contrary to Lemma~\ref{lemma-conn}.

\begin{figure}[!t]
\begin{center}
\includegraphics[width=60mm]{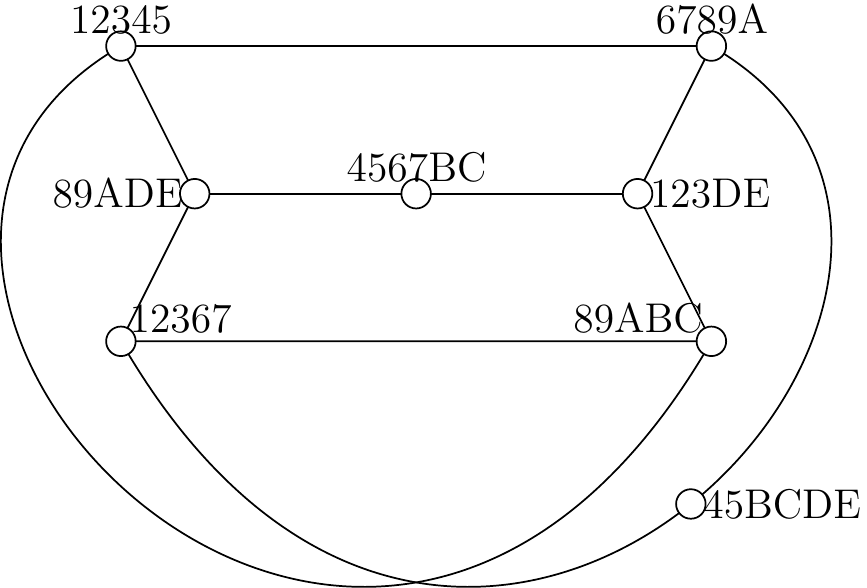}
\end{center}
\caption{A configuration from Lemma~\ref{lemma-col2}.}\label{fig-col2}
\end{figure}

Let us now consider the case where $H$ contains exactly four vertices of degree three in $G$.
By symmetry, we
can assume that $a,b,c\in V(H)$ (let us note that $a$, $b$, $c$ and $d$ are pairwise distinct, as three of them belong to $H$
and have degree exactly two).  Let $u$ be the vertex of $H$ distinct from $a$, $b$ and $c$ that is safe in $G$.
If $u\neq d$, then the edge $yd$ together with an edge incident with $u$ form
a $2$-edge-cut in $G$. Thus, Lemma~\ref{lemma-3conn} implies that
$d$ has degree two in $G$, contrary to Lemma~\ref{lemma-no2t2}.  It follows that $u=d$ and
$G$ is the graph depicted in Figure~\ref{fig-col2}.  However, $G$ is then
$(f_\varnothing,14)$-colorable, which is a contradiction.

\ref{it:2}. Suppose, on the contrary, that $a=c$.  Let $G'=G-v$.  As argued,
$\varnothing$ is a nail for $G'$, so the minimality of $G$ ensures the
existence of an $(f^{G'}_\varnothing,14t)$-coloring $\psi'$ of $G'$ for a
positive integer $t$.  Note that
$f^{G'}_\varnothing(x)=6/14=f^{G'}_\varnothing(y)$, while $f^G_\varnothing(x)=5/14=f^G_\varnothing(y)$.  Let $M$ be an arbitrary
subset of $\sset{14t}\setminus \psi'(a)$ of size $t$.  Define a coloring $\psi$ of $G$ as follows. For $z\in V(G)\setminus \{x,v,y\}$,
set $\psi(z)=\psi'(z)$; furthermore, set $\psi(x)=\psi'(x)\setminus M$, $\psi(y)=\psi'(y)\setminus M$ and $\psi(v)=\psi'(a)\cup M$.
Then $\psi$ is an $(f^G_\varnothing,14t)$-coloring of $G$, which is a contradiction.

\ref{it:3}. Again, let $G'=G-v$ and let $\psi'$ be an $(f^{G'}_\varnothing,14t)$-coloring of $G'$.
As $\abs{\psi'(x)}=6t=\abs{\psi'(y)}$,
there exists a subset $M$ of $\sset{14t}\setminus (\psi'(x)\cup \psi'(y))$ of size $2t$.
Let $S_a=\psi'(a)\setminus (\psi'(b)\cup M)$,
$S_b=\psi'(b)\setminus (\psi'(a)\cup M)$ and $S_{ab}=(\psi'(a)\cap\psi'(b))\setminus M$; note that $3t\le \abs{S_a}+\abs{S_{ab}}\le 5t$ and $3t\le \abs{S_b}+\abs{S_{ab}}\le 5t$.
Furthermore, since $\psi'(x)$ has size $6t$ and is disjoint from
$M\cup\psi'(a)\cup\psi'(b)$, it follows that
$14t-\abs{M}-\abs{S_a}-\abs{S_{ab}}-\abs{S_b}\ge 6t$, i.e.,
$\abs{S_a}+\abs{S_{ab}}+\abs{S_b}\le 6t$.
Our next goal is to choose a set $X$ in $\sset{14t}\setminus M$ of size $8t$
such that $\abs{X\cap \psi'(a)}\le t$ and $\abs{X\cap \psi'(b)}\le t$.
To this end, we consider several cases, regarding the sizes of $S_a$ and $S_b$.
If $\abs{S_a}\ge t$ and $\abs{S_b}\ge t$,
then choose $X$ so that $\abs{X\cap S_a}=\abs{X\cap S_b}=t$ and $X\cap S_{ab}=\varnothing$.  Otherwise, by symmetry, we can assume that $\abs{S_a}<t$; consequently,
$\abs{S_{ab}}\ge 3t-\abs{S_a}>2t$.  If $\abs{S_b}\ge t$, then let $X$ consist of $7t$ elements of $\sset{14t}\setminus (M\cup \psi'(b))$ and $t$ elements of $S_b$.
Finally, if both $S_a$ and $S_b$ have less than $t$ elements, supposing $\abs{S_a}\le\abs{S_b}<t$, then let $X$ consist of $S_a\cup S_b$ together with
$t-\abs{S_b}$ elements of $S_{ab}$ and $8t-\abs{S_a}-\abs{S_b}-(t-\abs{S_b})=7t-\abs{S_a}$ elements of $\sset{14t}\setminus (M\cup \psi'(a)\cup \psi'(b))$; this is possible,
since $\abs{\sset{14t}\setminus (M\cup \psi'(a)\cup
\psi'(b))}=12t-\abs{S_a}-(\abs{S_b}+\abs{S_{ab}})\ge 7t-\abs{S_a}$.  In each
case,
$\abs{X\cap \psi'(z)}\le t$ for $z\in\{a,b\}$, as desired.
Symmetrically, there exists a set $Y$ in $\sset{14t}\setminus M$ such that
$\abs{Y\cap\psi'(z)}\le t$ for $z\in\{c,d\}$.

An $(\ovf,14t)$-coloring of $G$ is now obtained as follows. Set $\psi(z)=\psi'(z)$ for $z\in V(G)\setminus \{a,b,c,d,v,x,y\}$,
$\psi(x)=X$, $\psi(v)=M$, $\psi(y)=Y$, $\psi(a)=\psi'(a)\setminus X$, $\psi(b)=\psi'(b)\setminus X$, $\psi(c)=\psi'(c)\setminus Y$ and $\psi(d)=\psi'(d)\setminus Y$.
\end{proof}

\begin{lemma}\label{lemma-3reg}
Every minimal counterexample to Theorem~\ref{thm-maingen} is $3$-regular.
\end{lemma}
\begin{proof}
Suppose, on the contrary, that $G$ is a minimal counterexample containing a
vertex $v$ of degree two.  By Lemmas~\ref{lemma-no22} and \ref{lemma-no2t2},
all the other
vertices of $G$ at distance at most two from $v$ have degree three.  Let $x$ and $y$ be the neighbors of $v$; let the neighbors of $x$ distinct from $v$ be $a$ and $b$,
and let the neighbors of $y$ distinct from $v$ be $c$ and $d$.  By Lemma~\ref{lemma-col2}, the vertices $a$, $b$, $c$ and $d$ are pairwise distinct.

In order to obtain a contradiction, we show that $G$ is
$f_\varnothing$-colorable.  To do so, we use the equivalent statement
given by Theorem~\ref{thm-eqchar}(d).
Let us consider an arbitrary non-negative weight function $w$ for $G$.  We need to show that $G$ contains an independent set $X$ with $w(X)\ge w_{f_\varnothing}$.
Let $w_2=w(a)+w(b)+w(c)+w(d)$.

\begin{assertion}
$G$ contains an independent set $X_0$ satisfying
\[w(X_0)\ge w_{f_\varnothing} + \frac{1}{14}(w_2-3w(x)-3w(y)+4w(v)).\]
\end{assertion}
To prove Assertion~*,
we discuss several cases depending on the values of $w$ on vertices at distance at most two from $v$.
By symmetry, we assume that $w(x)\le w(y)$.  Let $G'=G-\{v,x,y\}$, and recall that
$\varnothing$ is a nail for $G'$ by Lemma~\ref{lemma-col2}.

Suppose first that $w(y)\le w(v)$.
Note that
\begin{align*}
w_{f^{G'}_\varnothing}&=
\begin{aligned}[t]
w_{f^{G}_\varnothing}+\frac{1}{14}(6w(a)-5w(a)&+6w(b)-5w(b)+6w(c)-5w(c)\\
&+6w(d)-5w(d)-5w(x)-5w(y)-6w(v))
\end{aligned}\\
&=w_{f^{G}_\varnothing}+\frac{1}{14}(w_2-5w(x)-5w(y)-6w(v)).
\end{align*}
By the minimality of $G$,
there exists an independent set $P$ of $G'$ with $w(P)\ge
w_{f^{G'}_\varnothing}$.  Let $X_0=P\cup \{v\}$ and note that $X_0$ is an
independent set of $G$ such that
\begin{align*}
w(X_0)&=  w(P)+w(v) \ge w_{f^{G'}_\varnothing}+w(v)\\
&= w_{f^{G}_\varnothing}+\frac{1}{14}(w_2-5w(x)-5w(y)-6w(v))+w(v)\\
&= w_{f^{G}_\varnothing}+\frac{1}{14}(w_2-5w(x)-5w(y)+8w(v))\\
&=\begin{aligned}[t]
  w_{f^{G}_\varnothing}+\frac{1}{14}(w_2-3w(x)-3w(y)+4w(v))+ \frac{2}{14}(w(v)&-w(x))\\
 &  + \frac{2}{14}(w(v)-w(y))
 \end{aligned}\\
&\ge w_{f^{G}_\varnothing}+\frac{1}{14}(w_2-3w(x)-3w(y)+4w(v)).
\end{align*}

Next, suppose that $w(x)\le w(v)<w(y)$.  Let $w'$ be the (not necessarily non-negative) weight function
defined as follows: set $w'(z)=w(z)$ for $z\in V(G)\setminus \{c,d,v,x,y\}$, $w'(c)=w(c)-w(y)+w(v)$ and $w'(d)=w(d)-w(y)+w(v)$.
Note that
\begin{align*}
w'_{f^{G'}_\varnothing}&=\begin{aligned}[t]
w_{f^{G}_\varnothing}+\frac{1}{14}(6w'(a)-5w(a)&+6w'(b)-5w(b)+6w'(c)-5w(c)\\
&+6w'(d)-5w(d)-5w(x)-5w(y)-6w(v))
\end{aligned}\\
&=w_{f^{G}_\varnothing}+\frac{1}{14}(w_2-5w(x)-17w(y)+6w(v)).
\end{align*}
By the minimality of $G$ and Theorem~\ref{thm-eqchar}(c),
there exists an independent set $P$ of $G'$ with $w'(P)\ge w'_{f^{G'}_\varnothing}$.  Let $X_0$ be defined as follows: if $\{c,d\}\cap P\neq \varnothing$, then
let $X_0=P\cup\{v\}$, otherwise let $X_0=P\cup \{y\}$.  In the latter case,
$w(X_0)=w'(P)+w(y)$. In the former case (supposing $c\in P$),
it holds that $w(X_0)\ge w'(P)+(w(c)-w'(c))+w(v)=w'(P)+w(y)$ (the inequality holds, since if $d$ also belongs to $P$, then the right side changes by $w(d)-w'(d)=w(y)-w(v)>0$).
It follows that
\begin{align*}
w(X_0)&\ge w'(P)+w(y)\ge w'_{f^{G'}_\varnothing}+w(y)\\
&=w_{f^{G}_\varnothing}+\frac{1}{14}(w_2-5w(x)-17w(y)+6w(v))+w(y)\\
&=w_{f^{G}_\varnothing}+\frac{1}{14}(w_2-5w(x)-3w(y)+6w(v))\\
&=w_{f^{G}_\varnothing}+\frac{1}{14}(w_2-3w(x)-3w(y)+4w(v)+2(w(v)-w(x))\\
&\ge w_{f^{G}_\varnothing}+\frac{1}{14}(w_2-3w(x)-3w(y)+4w(v)).
\end{align*}

Finally, assume that $w(v)<w(x)\le w(y)$.
Let $w'$ be the (not necessarily non-negative) weight function
defined as follows: set $w'(z)=w(z)$ for $z\in V(G)\setminus \{a,b,c,d,v,x,y\}$, $w'(a)=w(a)-w(x)+w(v)$, $w'(b)=w(b)-w(x)+w(v)$, $w'(c)=w(c)-w(y)+w(v)$ and $w'(d)=w(d)-w(y)+w(v)$.
Note that
\begin{align*}
w'_{f^{G'}_\varnothing}&=\begin{aligned}[t]
w_{f^{G}_\varnothing}+\frac{1}{14}(6w'(a)-5w(a)&+6w'(b)-5w(b)+6w'(c)-5w(c)\\
&+6w'(d)-5w(d)-5w(x)-5w(y)-6w(v))
\end{aligned}\\
&=w_{f^{G}_\varnothing}+\frac{1}{14}(w_2-17w(x)-17w(y)+18w(v)).
\end{align*}
By the minimality of $G$ and Theorem~\ref{thm-eqchar}(c),
there exists an independent set $P$ of $G'$ with $w'(P)\ge w'_{f^{G'}_\varnothing}$.
We now show that there exists an independent set $X_0$ of $G$ such that
$w(X_0)\ge w'(P)+w(x)+w(y)-w(v)$. Indeed,
if $\{a,b\}\cap P\neq \varnothing$ and $\{c,d\}\cap P\neq \varnothing$
(supposing $a\in P$ and $c\in P$), then set $X_0=P\cup \{v\}$. It follows that
\begin{align*}
w(X_0)&\ge w'(P)+(w(a)-w'(a))+(w(c)-w'(c))+w(v)\\
&=w'(P)+w(x)+w(y)-w(v),
\end{align*}
as wanted. If $\{a,b\}\cap P\neq \varnothing$ (supposing $a\in P$) and $\{c,d\}\cap P=\varnothing$,
then let $X_0=P\cup \{y\}$. It follows that $w(X_0)\ge
w'(P)+(w(a)-w'(a))+w(y)=w'(P)+w(x)+w(y)-w(v)$, as wanted.  Similarly, if $\{a,b\}\cap P=\varnothing$ and $\{c,d\}\cap P\neq \varnothing$, then let $X_0=P\cup \{x\}$ and observe that
$w(X_0)\ge w'(P)+w(x)+w(y)-w(v)$. Last, if $\{a,b\}\cap P=\varnothing$ and
$\{c,d\}\cap P=\varnothing$, then let $X_0=P\cup \{x,y\}$. It follows that $w(X_0)=w'(P)+w(x)+w(y)\ge w'(P)+w(x)+w(y)-w(v)$.
In conclusion,
\begin{align*}
w(X_0)&\ge w'(P)+w(x)+w(y)-w(v)\\
&\ge w'_{f^{G'}_\varnothing}+w(x)+w(y)-w(v)\\
&= w_{f^{G}_\varnothing}+\frac{1}{14}(w_2-17w(x)-17w(y)+18w(v))+w(x)+w(y)-w(v)\\
&= w_{f^{G}_\varnothing}+\frac{1}{14}(w_2-3w(x)-3w(y)+4w(v)).
\end{align*}

Therefore, in all the cases the set $X_0$ has the required weight.
This concludes the proof of Assertion~*.

By Lemma~\ref{lemma-col2}, the graph
$G$ has an $(\ovf, 14t)$-coloring $\psi$ for a positive integer $t$.
For $i\in\sset{14t}$, let $X_i=\sst{z\in V(G)}{i\in\psi(z)}$; note that $X_i$ is an independent set
of $G$ and
\begin{align*}
\frac{1}{14t}\sum_{i=1}^{14t} w(X_i)&=\sum_{z\in V(G)} \ovf(z)w(z)\\
&=\sum_{z\in \{a,b,c,d,x,y,v\}}(\ovf(z)-f_\varnothing(z))w(z)+\sum_{z\in V(G)} f_\varnothing(z)w(z)\\
&=\frac{1}{14}(-w_2+3w(x)+3w(y)-4w(v))+\sum_{z\in V(G)} f_\varnothing(z)w(z)\\
&= w_{f_\varnothing}+\frac{1}{14}(-w_2+3w(x)+3w(y)-4w(v)).
\end{align*}

Together with Assertion~*, this implies that $w(X_i)\ge w_{f_\varnothing}$ for some
$i\in\iset{0}{14t}$.
Since this holds for every non-negative weight function for $G$, we conclude that
$G$ has an $f^{G}_\varnothing$-coloring, which is a contradiction.
\end{proof}

\begin{lemma}\label{lemma-girth5}
Every minimal counterexample to Theorem~\ref{thm-maingen} has girth at least five.
\end{lemma}
\begin{proof}
Suppose, on the contrary, that $G$ is a minimal counterexample that contains a $4$-cycle $uvxy$.
Let $a$, $c$, $b$ and $d$ be the neighbors of $u$, $v$, $x$ and $y$, respectively, outside this $4$-cycle.

Since $G$ is triangle-free, $\{a,b\}\cap\{c,d\}=\varnothing$.
If $a=b$, then $u$ and $x$ have the same neighborhood in $G$ but they are not
adjacent. The set $B=\{a,v,y\}$ being a nail for $G-u$,
the minimality of $G$ implies that $G-u$ has an $f_B$-coloring $\psi$.
Setting $\psi(u)=\psi(x)$ yields
an $f_\varnothing$-coloring of $G$, which is a contradiction.

Therefore, $a\neq b$ and, symmetrically, $c\neq d$. It follows that $a$, $b$, $c$ and
$d$ are pairwise distinct.  Let $G'=G-\{u,v,x,y\}$.  Consider a dangerous
induced subgraph $H$ of $G'$. Lemma~\ref{lemma-nokp4} implies that
$H$ is a $5$-cycle.  Furthermore, by Lemma~\ref{lemma-conn}, not all of $a$, $b$, $c$ and $d$ belong to $V(H)$, as otherwise
the vertex of $H$ distinct from $a$, $b$, $c$ and $d$ would be incident with a bridge.  Therefore, $H$ contains at least two vertices
of degree three in $G'$.  It follows that $\varnothing$ is a nail for $G'$.
By the minimality of $G$, there exists an $(f^{G'}_\varnothing,14t)$-coloring
$\psi'$ of $G'$ for a positive integer $t$.
Let $A_u$, $A_v$, $A_x$, $A_y$, $B_u$, $B_v$, $B_x$ and $B_y$ be defined in the same way as in the proof of Lemma~\ref{lemma-nokp4}.
Let $\psi$ be the coloring of $G$ defined by
$\psi(z)=\psi'(z)$ for $z\in V(G)\setminus \{a,b,c,d,u,v,x,y\}$, $\psi(a)=\psi'(a)\setminus B_u$, $\psi(b)=\psi'(b)\setminus B_x$, $\psi(c)=\psi'(c)\setminus B_v$, $\psi(d)=\psi'(d)\setminus B_y$,
$\psi(u)=A_u\cup B_u$, $\psi(v)=A_v\cup B_v$, $\psi(x)=A_x\cup B_x$ and $\psi(y)=A_y\cup B_y$.  Then $\psi$ is an $(f^G_\varnothing,14t)$-coloring of $G$,
which is a contradiction.
\end{proof}

Finally, we are ready to prove our main result.

\begin{proof}[of Theorem~\ref{thm-maingen}]
If Theorem~\ref{thm-maingen} were false, there would exist a subcubic triangle-free graph $G$ with a nail $B$
forming a minimal counterexample to Theorem~\ref{thm-maingen}.  Then,
Lemma~\ref{lemma-nob} implies that $B=\varnothing$,
while Lemmas~\ref{lemma-3reg} and \ref{lemma-girth5} yield that $G$ is $3$-regular and contains no $4$-cycles.

Let $w$ be any non-negative weight function for $G$.  For $u,v\in V(G)$, let
$d(u,v)$ be the length of a shortest
path between $u$ and $v$.  For a vertex $v\in V(G)$, let
\[W_v=9w(v)-5\sum_{u\colon d(u,v)=1} w(u)+\sum_{u\colon d(u,v)=2} w(u).\]
Since $G$ is $3$-regular and has girth at least five, for each $u\in V(G)$, there are exactly three vertices $v$ with $d(u,v)=1$ and exactly six vertices
with $d(u,v)=2$; consequently,
\begin{align*}
\sum_{v\in V(G)} W_v&=9\sum_{v\in V(G)} w(v) - 5\sum_{v\in V(G)}\sum_{u\colon d(u,v)=1} w(u)+\sum_{v\in V(G)}\sum_{u\colon d(u,v)=2} w(u)\\
&=9\sum_{v\in V(G)} w(v) - 5\sum_{u\in V(G)}\sum_{v\colon d(u,v)=1} w(u)+\sum_{u\in V(G)}\sum_{v\colon d(u,v)=2} w(u)\\
&=9\sum_{v\in V(G)} w(v) - 5\sum_{u\in V(G)}3w(u)+\sum_{u\in V(G)}6w(u)\\
&=(9-15+6)\sum_{v\in V(G)} w(v)\\
&=0.
\end{align*}

Therefore, there exists a vertex $v\in V(G)$ such that $W_v\ge 0$.  Let $u_1$, $u_2$ and $u_3$ be the neighbors of $v$, and
let $x_1, \ldots, x_6$ be the six vertices of $G$ at distance exactly $2$
from $v$.
Set $G'=G-\{v,u_1,u_2,u_3\}$.  Consider a dangerous induced subgraph $H$ of
$G$.  By Lemma~\ref{lemma-nokp4}, we know that
$H$ is a $5$-cycle.  Let $S=V(H)\cap \{x_1,\ldots, x_6\}$.  If $\abs{S}\ge 4$, then at least two of the vertices in $S$ have a common
neighbor among $u_1$, $u_2$ and $u_3$. By symmetry, assume that $u_1$ is
adjacent to both $x_1$ and $x_2$.  Since $G$ is triangle-free, $x_1$ is not
adjacent to $x_2$, and thus these two vertices also have a common neighbor in
$H$.  Consequently, $G$ contains a $4$-cycle, which is a contradiction.
Therefore, each dangerous induced subgraph of $G'$ contains at least two special vertices of degree three.
It follows that $\varnothing$ is a nail for $G'$.

Note that
\begin{align*}
w_{f^{G'}_\varnothing}&=w_{f^{G}_\varnothing}+\frac{1}{14}\left(6\sum_{i=1}^6w(x_i)-5\sum_{i=1}^6w(x_i)-5\sum_{i=1}^3w(u_i)-5w(v)\right)\\
&= w_{f^{G}_\varnothing}+\frac{1}{14}(W_v-14w(v)).
\end{align*}
By the minimality of $G$ and Theorem~\ref{thm-eqchar}, there exists an independent set $P$ of $G'$ such that
$w(P)\ge w_{f^{G'}_\varnothing}$.  Let $X=P\cup \{v\}$.  Then
\begin{align*}
w(X)&=w(P)+w(v)\\
&\ge w_{f^{G'}_\varnothing}+w(v)\\
&=w_{f^{G}_\varnothing}+\frac{1}{14}(W_v-14w(v))+w(v)\\
&=w_{f^{G}_\varnothing}+\frac{1}{14}W_v\\
&\ge w_{f^{G}_\varnothing}.
\end{align*}

Therefore, for every non-negative weight function $w$ for $G$, there exists an independent set $X$
of $G$ such that $w(X)\ge w_{f^{G}_\varnothing}$.  By Theorem~\ref{thm-eqchar}, we conclude that
$G$ has an $f^{G}_\varnothing$-coloring.  This is a contradiction, showing that there exists
no counterexample to Theorem~\ref{thm-maingen}.
\end{proof}

\section{Conclusion}
We believe that the method developed in this paper may be relevant for other fractional colouring
problems, and in particular for Conjecture~\ref{conjecture-planar}.
However, a straightforward attempt to combine our ideas with those of
Heckman and Thomas~\cite{HeTh06} fails, since they use the integrality of the independence number which permits to
round up the obtained lower bounds.

In order to prove Theorem~\ref{thm-maingen}, we used several equivalent definitions of (weighted) fractional colorings.
As a consequence, our proof is not constructive and the following question is open.
\begin{problem}\label{prob-1}
Does there exist a polynomial-time algorithm to find a fractional $14/5$-coloring of a given input subcubic triangle-free graph?
\end{problem}
We pause here to note that, in general, even if a graph is known to have
fractional chromatic number at most $r$ and, thus,
an $(rN:N)$-coloring for some integer $N$, it is not even clear whether such a coloring can be written
in polynomial space. Indeed, all such values of $N$ may be exponential in the number
of vertices, as is the case, e.g., for the Mycielski graphs~\cite{LPU95}.  This issue would be avoided if the answer to the
following question is positive.

\begin{problem}\label{prob-3}
Does there exist an integer $t$ such that every subcubic triangle-free graph has a $(14t:5t)$-coloring?
\end{problem}

\bibliographystyle{amsplain}
\bibliography{fracsub}
\end{document}